\documentclass{mcom-l}

\usepackage[english]{babel}
\usepackage{amssymb}
\usepackage{amsmath}
\usepackage{psfrag}
\usepackage{graphicx}
\usepackage{caption}

\newtheorem{theorem}{Theorem}[section]
\newtheorem{lemma}[theorem]{Lemma}
\theoremstyle{definition}
\newtheorem{definition}[theorem]{Definition}
\theoremstyle{remark}
\newtheorem{remark}[theorem]{Remark}
\theoremstyle{proposition}
\newtheorem{proposition}[theorem]{Proposition}
\theoremstyle{corollary}
\newtheorem{corollary}[theorem]{Corollary}

\numberwithin{equation}{section}

\newcommand{\abs}[1]{\left\vert #1 \right\vert}
\newcommand{\norm}[1]{\left\| #1 \right\|}
\newcommand{\lefttriplenorm}{\ensuremath{\left| \! \left| \! \left|}}
\newcommand{\righttriplenorm}{\ensuremath{\right| \! \right| \! \right|}}
\newcommand{\triplenorm}[1]{\lefttriplenorm #1 \righttriplenorm}
\newcommand{\skp}[1]{\left< #1 \right>}

\begin{document}

\title[$\mathcal{H}$-matrix approximation to the inverses of BEM matrices]
{Existence of $\mathcal{H}$-matrix approximants to the inverses of BEM matrices: 
the simple-layer operator}


\author{Markus Faustmann}
\address{Institute for Analysis and Scientific Computing (Inst. E 101),
Vienna University of Technology,
Wiedner Hauptstraße 8-10,
1040 Wien, Austria }
\email{markus.faustmann\@ tuwien.ac.at}

\author{Jens Markus Melenk}
\address{Institute for Analysis and Scientific Computing (Inst. E 101),
Vienna University of Technology,
Wiedner Hauptstraße 8-10,
1040 Wien, Austria }
\email{melenk\@ tuwien.ac.at}

\author{Dirk Praetorius}
\address{Institute for Analysis and Scientific Computing (Inst. E 101),
Vienna University of Technology,
Wiedner Hauptstraße 8-10,
1040 Wien, Austria }
\email{dirk.praetorius\@ tuwien.ac.at}

\subjclass[2010]{Primary 65F05; Secondary 65N38, 65F30, 65F50}

\date{}

\dedicatory{Dedicated to Wolfgang Hackbusch on the occasion of his 65th birthday}

\begin{abstract}
We consider the question of approximating 
the inverse $\mathbf W = \mathbf V^{-1}$ of the Galerkin stiffness
matrix $\mathbf V$ obtained by discretizing the simple-layer operator $V$ 
with piecewise constant functions. The block partitioning of $\mathbf W$ 
is assumed to satisfy any one of several standard admissibility criteria that 
are employed in connection with clustering algorithms to approximate 
the discrete BEM operator $\mathbf V$. We show that $\mathbf W$ can be approximated
by blockwise low-rank matrices such that the error  decays exponentially
in the block rank employed. Similar exponential approximability results are 
shown for the Cholesky factorization 
of $\mathbf V$.
\end{abstract}

\maketitle

\section{Introduction}
The system matrices arising in the boundary element method (BEM) such as the 
matrix $\mathbf V$ for the classical simple-layer operator $V$ are fully populated.
The classical BEM is therefore often deemed inefficient with respect to memory 
requirements and, in turn, fast iterative solvers. 
These reservations can be met with various compression techniques 
that have been developed in the past to 
store the BEM matrices and realize the matrix-vector multiplication 
with log-linear (or even linear) complexity. We mention here 
multipole expansions, \cite{rokhlin85,greengard-rokhlin97}, 
panel clustering, \cite{hackbusch-nowak88,hackbusch-nowak89,hackbusch-sauter93,sauter92}, 
wavelet compression techniques, 
\cite{rathsfeld98,rathsfeld01,schneider98,petersdorff-schwab-schneider97,tausch03,tausch-white03},
the mosaic-skeleton method, \cite{tyrtyshnikov00}, the adaptive cross approximation (ACA)
method, \cite{bebendorf00}, and the hybrid cross approximation (HCA), \cite{boerm-grasedyck05}. 
Many of these data-sparse methods can be understood as specific instances of 
${\mathcal H}$-matrices, which were introduced and analyzed
in \cite{Hackbusch99,GrasedyckHackbusch,GrasedyckDissertation,HackbuschBuch} 
as blockwise low rank matrices.
$\mathcal{H}$-matrices come with the additional feature that they permit 
an (approximate) arithmetic with log-linear complexity. In particular, this arithmetic 
includes the (approximate) inversion of matrices. Thus, the 
${\mathcal H}$-matrix arithmetic can provide an approximation to the inverse, 
the $LU$ factorization, or the Cholesky decomposition. However, the 
accuracy of this approximate inverse or factorization depends 
on various parameters including the rank of the blocks and is,
for the matrices arising from BEM, mathematically not fully understood. 
In the present paper we show that for a block structure typically
employed in the context of ${\mathcal H}$-matrices,  the inverse ${\mathbf W} = \mathbf{V}^{-1}$ of 
the discretization of the simple-layer operator $V$ can be approximated from 
the set of blockwise rank-$r$ matrices at an exponential rate in the block rank. 
While this result does not fully analyze the accuracy of the 
${\mathcal H}$-matrix inversion algorithms, it shows that inversion algorithms within
the ${\mathcal H}$-matrix framework could work. It thus
gives some mathematical underpinning to the success of the 
${\mathcal H}$-matrix calculus when employed to compute (approximate) 
inverses of BEM-matrices, which is observed numerically, for example,  
by Bebendorf~\cite{Bebendorf05} and Grasedyck~\cite{GrasedyckDissertation,Grasedyck05}. 

Quickly after the introduction of the $\mathcal{H}$-matrix arithmetic, also 
$\mathcal{H}$-factorizations mimicking the classical $LU$- and Cholesky decompositions
were proposed, \cite{Lintner,Bebendorf05}. Again, 
numerical evidence indicates their great usefulness for example, for
black box preconditioning in iterative solvers,
\cite{Bebendorf05,Grasedyck05,Grasedyck08,leborne-grasedyck06,grasedyck-kriemann-leborne08}.

The class of ${\mathcal H}$-matrices is not the only one for which 
inversion and factorizations of system matrices arising in the 
discretization of differential and integral operators have been devised. 
Closely related to the concept of ${\mathcal H}$-matrices and its arithmetic are 
``hierarchically semiseparable matrices'', 
\cite{xia13,xia-chandrasekaran-gu-li09,li-gu-wu-xia12} and the idea of 
``recursive skeletonization'', 
\cite{ho-greengard12,greengard-gueyffier-martinsson-rokhlin09,ho-ying13};
for discretizations of PDEs, we mention \cite{ho-ying13,gillman-martinsson13,schmitz-ying12,martinsson09}, and  
particular applications to boundary integral 
equations are \cite{martinsson-rokhlin05,corona-martinsson-zorin13,ho-ying13a}. 
These factorization algorithms aim to exploit that some off-diagonal blocks of 
certain Schur complements are low rank. 
Following 
\cite{Bebendorf07,GrasedyckKriemannLeBorne,chandrasekaran-dewilde-gu-somasuderam10}, we rigorously establish 
that the off-diagonal blocks of certain Schur complements 
can be approximated by low-rank matrices. We exploit this fact to show that 
the Cholesky decomposition of ${\mathbf V}$ can be approximated at an 
exponential rate in the block rank in the ${\mathcal H}$-matrix format. 

Hitherto, the mathematical analysis of approximability of the inverse
of system matrices in the ${\mathcal H}$-matrix format has focused 
on the setting of the finite element method (FEM). The first result 
in this direction is due to \cite{BebendorfHackbusch}.
Generalizations to elliptic systems \cite{schreittmiller06} and 
approximations in the framework of $\mathcal{H}^2$-matrices \cite{Boerm,BoermBuch} are also analyzed.
Our recent works \cite{FaustmannMelenkPraetorius,FMP13}
differ from the above mentioned references for ${\mathcal H}$-matrices 
for FEM-matrices in several ways. Among the 
differences, we highlight that, as in present paper,  
\cite{FaustmannMelenkPraetorius,FMP13} work in a fully discrete setting in contrast
to the earlier technique of approximating on the continuous level and then 
projecting into discrete spaces. This technique avoids the projection error
associated with the transition from the continuous level to the discrete one, 
and leads to exponential convergence in the block rank. 

In the present paper, we focus on the lowest-order discretization of 
the simple-layer operator associated with the Laplace operator. 
However, our arguments are based on rather general properties of elliptic 
operators so that similar assertions can be shown
for higher order discretizations and the hypersingular integral equation,
which can be found in the forthcoming work \cite{FMP14}. Moreover,
we expect that our approach can cover the case of elliptic systems amenable to 
a treatment with the BEM such as the Lam\'e system. 

The paper is structured as follows. In Section~\ref{sec:main-results}, 
we present the main results for ${\mathcal H}$-matrices. Mathematically, 
the core of the paper is Section~\ref{sec:Approximation-solution}, where we investigate
the question of how well solutions of the discrete system can be approximated
locally from low dimensional spaces. These results are transferred to the matrix
level in Sections~\ref{sec:H-matrix-approximation} and \ref{sec:LU-decomposition}  
to show the approximability result for ${\mathbf V}^{-1}$ and for Cholesky decompositions, 
respectively. 
Section~\ref{sec:extensions} is concerned with various extensions of our 
approximation result: 
We show a similar compression result for the 
Poincar\'e-Steklov operator and we show that ${\mathbf V}^{-1}$ can be
approximated at an exponential rate in the format of 
${\mathcal H}^2$-matrices, \cite{HackbuschKhoromskijSauter,Boerm,BoermBuch}. 
\medskip

\section{Main Result}
\label{sec:main-results}

Let $\Omega \subset \mathbb{R}^{d}$, $d \in \{2,3\}$, be a bounded Lipschitz domain such that $\Gamma:=\partial \Omega$ 
is polygonal (for $d = 2$) or polyhedral (for $d = 3$).
We consider the simple-layer operator $V \in L(H^{-1/2}(\Gamma),H^{1/2}(\Gamma))$ associated with the Laplacian
 given by 
\begin{equation*}
V\phi(x) = \int_{\Gamma} G(x-y)\phi(y)ds_y, \quad x \in \Gamma,
\end{equation*}
where $G(x) = -\frac{1}{2\pi} \log\abs{x}$ for $d=2$ and $G(x) = \frac{1}{4\pi}\frac{1}{\abs{x}}$ for $d=3$ 
is the fundamental solution of the Laplacian. 
The simple-layer operator is an elliptic isomorphism for $d=3$ and for $d=2$ provided $\operatorname*{diam}(\Omega) < 1$, 
which can be assumed by scaling. We refer the reader to the monographs
\cite{mclean00,hsiao-wendland08,SauterSchwab,Steinbach} for a detailed 
discussion of the pertinent properties of boundary integral operators such as 
the simple-layer operator studied here. 

We assume that $\Gamma$ is triangulated by a {\it quasiuniform} mesh ${\mathcal T}_h=\{T_1,\dots,T_N\}$ of 
mesh width $h := \max_{T_j\in \mathcal{T}_h}{\rm diam}(T_j)$. 
The elements $T_j \in \mathcal{T}_h$ are open line segments ($d=2$) or triangles ($d=3$). 
Additionally, we assume that the mesh $\mathcal{T}_h$ is regular in the sense of Ciarlet and 
$\gamma$-shape regular in the sense that for $d=2$ the quotient of the diameters of neighboring elements 
is bounded by $\gamma$ and for $d=3$ we have ${\rm diam}(T_j) \le \gamma\,|T_j|^{1/2}$ for all $T_j\in\mathcal{T}_h$. 
In the following, the notation $\lesssim$ abbreviates $\leq$ up to a 
constant $C>0$ which depends only on $\Omega$, the dimension $d$, and the
$\gamma$-shape regularity of the quasiuniform triangulation $\mathcal{T}_h$. 
Moreover, we use $\simeq$ to indicate that both estimates
$\lesssim$ and $\gtrsim$ hold.

We consider the lowest-order Galerkin discretization of $V$ by piecewise constant functions in 
$S^{0,0}({\mathcal T}_h) := \big\{u \in L^2(\Gamma)\, :\, u|_{T_j} \,\text{is constant} \, \forall T_j \in \mathcal{T}_h \big\}$. 
Throughout, we will work with the basis
${\mathcal B}_h:=\big\{\chi_j \,:\, j = 1,\dots, N\big\}$ of the space $S^{0,0}({\mathcal T}_h)$, 
where $\chi_j$ is the characteristic function associated with $T_j\in\mathcal T_h$.
With the isomorphism $\Phi:\mathbb{R}^N\rightarrow S^{0,0}({\mathcal T}_h)$, $\mathbf{x} \mapsto \sum_{j=1}^N\mathbf{x}_j\chi_j$, 
we note
\begin{equation}\label{eq:isomorphism}
h^{d/2}\norm{\mathbf{x}}_2 \lesssim \norm{\Phi(\mathbf{x})}_{L^2(\Gamma)} \lesssim h^{d/2}\norm{\mathbf{x}}_2
\quad \forall\, \mathbf{x} \in \mathbb{R}^d. 
\end{equation}

With the basis ${\mathcal B}_h$, the Galerkin discretization of $V$ leads to a symmetric and positive definite matrix
$\mathbf V\in\mathbb R^{N\times N}$, where 
\begin{equation} \label{eq:Galerkin} \mathbf V_{jk} = \langle V\chi_k,\chi_j\rangle = \int_{T_j}\int_{T_k}G(x-y)ds_y ds_x, 
\quad \chi_j,\chi_k\in {\mathcal B}_h,
\end{equation} 
and $\langle\cdot,\cdot\rangle$ denotes the $L^2(\Gamma)$-scalar product.

In the following, we study the approximability of the inverse BEM matrix 
$\mathbf{W} = \mathbf{V}^{-1}$ by some blockwise low-rank matrix $\mathbf{W}_{\mathcal{H}}$. 
First, we need to define the underlying block structure,  
which is based on the concept of ``admissibility'', 
introduced in the following definition. 
\begin{definition}[bounding boxes and $\eta$-admissibility]
\label{def:admissibility}
A \emph{cluster} $\tau$ is a subset of the index set $\mathcal{I} = \{1,\ldots,N\}$. For a cluster $\tau \subset \mathcal{I}$, 
we say that $B_{R_{\tau}} \subset \mathbb{R}^d$ 
is a \emph{bounding box} if: 
\begin{enumerate}
\def\theenumi{\roman{enumi}}
 \item \label{item:def:admissibility-i}
$B_{R_{\tau}}$ is a hyper cube with side length $R_{\tau}$,
 \item \label{item:def:admissibility-ii}
$ \overline{T_i} \subset B_{R_{\tau}}$ for all $ i \in \tau $,
\end{enumerate}

For $\eta > 0$, a pair of clusters $(\tau,\sigma)$ with $\tau,\sigma \subset \mathcal{I}$ 
is $\eta$-\emph{admissible} if 
there exist bounding boxes $B_{R_{\tau}}$, $B_{R_{\sigma}}$ satisfying 
(\ref{item:def:admissibility-i})--(\ref{item:def:admissibility-ii})
such that
\begin{equation}\label{eq:admissibility}
\min\{{\rm diam}(B_{R_{\tau}}),{\rm diam}(B_{R_{\sigma}})\} \leq  \eta \; {\rm dist}(B_{R_{\tau}},B_{R_{\sigma}}).
\end{equation}
\end{definition}

\begin{remark}
Since the operator $V$ is symmetric, we are able to use the admissibility condition \eqref{eq:admissibility} 
instead of the stronger admissibility condition
\begin{equation}\label{eq:maxadmissible}
\max\{{\rm diam}(B_{R_{\tau}}),{\rm diam}(B_{R_{\sigma}})\} \leq  \eta \; {\rm dist}(B_{R_{\tau}},B_{R_{\sigma}}),
\end{equation}
which is often encountered in clustering algorithms. 
This follows from the fact that Proposition~\ref{thm:function-approximation}
only needs an admissibility criterion of the form 
${\rm diam}(B_{R_{\tau}}) \leq \eta \, {\rm dist}(B_{R_{\tau}},B_{R_{\sigma}})$.
Due to the symmetry of $V$, deriving a block approximation for the block $\tau\times\sigma$
is equivalent to deriving an approximation for the block $\sigma\times\tau$. 
Therefore, we can interchange roles of the boxes $B_{R_{\tau}}$ and $B_{R_{\sigma}}$, 
and as a consequence the
weaker admissibility condition \eqref{eq:admissibility} is sufficient. 
\end{remark}

\begin{definition}[blockwise rank-$r$ matrices]
Let $P$ be a partition of ${\mathcal I} \times {\mathcal I}$ and $\eta>0$. 
A matrix ${\mathbf W}_{\mathcal{H}} \in \mathbb{R}^{N \times N}$ 
is said to be a \emph{blockwise rank-$r$ matrix}, if for every $\eta$-admissible cluster pair $(\tau,\sigma) \in P$, 
the block ${\mathbf W}_{\mathcal{H}}|_{\tau \times \sigma}$ is a rank-$r$ matrix, i.e., it has the form 
${\mathbf W}_{\mathcal{H}}|_{\tau \times \sigma} = {\mathbf X}_{\tau \sigma} {\mathbf Y}^T_{\tau \sigma}$ with 
$\mathbf{X}_{\tau\sigma} \in \mathbb{R}^{\abs{\tau}\times r}$ 
and $\mathbf{Y}_{\tau\sigma} \in \mathbb{R}^{\abs{\sigma}\times r}$.
Here and below, $\abs{\sigma}$ denotes the cardinality of a finite set $\sigma$.
\end{definition}

\subsection{Approximation of ${\mathbf V}^{-1}$}

The following Theorem~\ref{th:blockapprox} shows that 
admissible matrix blocks of $\mathbf{V}^{-1}$
can be approximated by rank-$r$ matrices and the error converges exponentially in the block rank.

\begin{theorem}\label{th:blockapprox}
Fix the admissibility parameter $\eta > 0$ and $q\in (0,1)$. Let the cluster pair $(\tau,\sigma)$ be $\eta$-admissible. 
Then, for every $k \in \mathbb{N}$, there are matrices
$\mathbf{X}_{\tau\sigma} \in \mathbb{R}^{\abs{\tau}\times r}$, 
$\mathbf{Y}_{\tau\sigma} \in \mathbb{R}^{\abs{\sigma}\times r}$ 
of rank $r \leq C_{\rm dim} (2+\eta)^d q^{-d}k^{d+1}$
such that
\begin{equation}
\norm{\mathbf{V}^{-1}|_{\tau \times \sigma} - \mathbf{X}_{\tau\sigma}\mathbf{Y}_{\tau\sigma}^T}_2 
\leq C_{\rm apx} N^{(d+2)/(d-1)} q^k.
\end{equation}
The constants $C_{\rm apx}$, $C_{\rm dim}>0$ depend only on $\Omega$, $d$, 
and the $\gamma$-shape regularity of the quasiuniform triangulation $\mathcal{T}_h$.
\end{theorem}
The approximation estimates for the individual blocks can be combined to assess the approximability 
of ${\mathbf V}^{-1}$ by blockwise rank-$r$ matrices. Particularly satisfactory estimates are obtained 
if the blockwise rank-$r$ matrices have additional structure. To that end, we introduce the following definitions. 

\begin{definition}[cluster tree]
A \emph{cluster tree} with \emph{leaf size} $n_{\rm leaf} \in \mathbb{N}$ is a binary tree $\mathbb{T}_{\mathcal{I}}$ with root $\mathcal{I}$  
such that for each cluster $\tau \in \mathbb{T}_{\mathcal{I}}$ the following dichotomy holds: either $\tau$ is a leaf of the tree and 
$\abs{\tau} \leq n_{\rm leaf}$, or there exist so called sons $\tau'$, $\tau'' \in \mathbb{T}_{\mathcal{I}}$, which are disjoint subsets of $\tau$ with 
$\tau = \tau' \cup \tau''$. The \emph{level function} ${\rm level}: \mathbb{T}_{\mathcal{I}} \rightarrow \mathbb{N}_0$ is inductively defined by 
${\rm level}(\mathcal{I}) = 0$ and ${\rm level}(\tau') := {\rm level}(\tau) + 1$ for $\tau'$ a son of $\tau$. The \emph{depth} of a cluster tree
is ${\rm depth}(\mathbb{T}_{\mathcal{I}}) := \max_{\tau \in \mathbb{T}_{\mathcal{I}}}{\rm level}(\tau)$.  
\end{definition}

\begin{definition}[far field, near field, and sparsity constant]
A partition $P$ of $\mathcal{I} \times \mathcal{I}$ is said to be based on the cluster tree $\mathbb{T}_{\mathcal{I}}$, 
if $P \subset \mathbb{T}_{\mathcal{I}}\times\mathbb{T}_{\mathcal{I}}$. For such a partition $P$ and fixed 
admissibility parameter $\eta > 0$, 
we define the \emph{far field} and the \emph{near field} as 
\begin{equation}\label{eq:farfield}
P_{\rm far} := \{(\tau,\sigma) \in P \; : \; (\tau,\sigma) \; \text{is $\eta$-admissible}\}, \quad P_{\rm near} := P\backslash P_{\rm far}.
\end{equation}
The \emph{sparsity constant} $C_{\rm sp}$, 
introduced in \cite{HackbuschKhoromskij2000a,HackbuschKhoromskij2000b,GrasedyckDissertation}, of such a partition is defined by
\begin{equation}\label{eq:sparsityConstant}
C_{\rm sp} := \max\left\{\max_{\tau \in \mathbb{T}_{\mathcal{I}}}\abs{\{\sigma \in \mathbb{T}_{\mathcal{I}} \, : \, \tau \times \sigma \in P_{\rm far}\}},
\max_{\sigma \in \mathbb{T}_{\mathcal{I}}}\abs{\{\tau \in \mathbb{T}_{\mathcal{I}} \, : \, \tau \times \sigma \in P_{\rm far}\}}\right\}.
\end{equation}
\end{definition}

The following Theorem~\ref{th:Happrox} shows that the matrix
$\mathbf{V}^{-1}$ can be approximated by blockwise rank-$r$ matrices at an exponential rate in the 
block rank $r$: 

\begin{theorem}\label{th:Happrox}
Fix the admissibility parameter $\eta > 0$. 
Let a partition $P$ of ${\mathcal{I}} \times {\mathcal{I}}$ be based on a cluster tree 
$\mathbb{T}_{\mathcal{I}}$. Then, there is a 
blockwise rank-$r$ matrix $\mathbf{W}_{\mathcal{H}}$ such that 
\begin{equation}
\norm{\mathbf{V}^{-1} - \mathbf{W}_{\mathcal{H}}}_2 \leq C_{\rm apx} C_{\rm sp} N^{(d+2)/(d-1)} {\rm depth}(\mathbb{T}_{\mathcal{I}}) e^{-br^{1/(d+1)}}.
\end{equation}
The constant $C_{\rm apx}$ depends only on $\Omega$, $d$, 
and the $\gamma$-shape regularity of the quasiuniform triangulation $\mathcal{T}_h$, 
while the constant $b>0$ additionally depends on $\eta$.
\end{theorem}
\begin{remark}
For quasiuniform meshes with $\mathcal{O}(N)$ elements,
typical clustering strategies such as the ``geometric clustering'' described in \cite{HackbuschBuch}
lead to fairly balanced cluster trees $\mathbb{T}_{\mathcal{I}}$ 
of depth $\mathcal{O}(\log N)$ and a sparsity constant $C_{\rm sp}$ that is bounded uniformly in $N$.
We refer to~\cite{HackbuschKhoromskij2000a,HackbuschKhoromskij2000b,GrasedyckDissertation,HackbuschBuch} for the fact that the
memory requirement to store $\mathbf{W}_{\mathcal{H}}$ is $\mathcal{O}\big((r+n_{\rm leaf}) N \log N\big)$.
\end{remark}

\begin{remark}
Using $h \simeq N^{-1/(d-1)}$ and 
$\frac{1}{\norm{\mathbf{V}^{-1}}_2} \leq \norm{\mathbf{V}}_2 \lesssim h^{(d-1)/2}\simeq N^{-1/2}$, 
where the last estimate can be found, e.g, in \cite[Lemma~12.6]{Steinbach}, we get a bound for the relative error
\begin{equation}
\frac{\norm{\mathbf{V}^{-1} - \mathbf{W}_{\mathcal{H}}}_2}{\norm{\mathbf{V}^{-1}}_2}
\lesssim C_{\rm apx} C_{\rm sp} N^{(d+5)/(2d-2)} {\rm depth}(\mathbb{T}_{\mathcal{I}}) e^{-br^{1/(d+1)}}.
\end{equation}
\end{remark}

\begin{remark}
The approximation result of Theorem~\ref{th:Happrox} is formulated in the spectral norm. In fact, 
inspection of the proof of Theorem~\ref{th:blockapprox} shows that we prove 
an approximation result in the weighted $L^2$-operator norm 
$\norm{\cdot}_{\sqrt{h}L^2 \rightarrow \frac {1}{\sqrt{h}}L^2}$. Other norms
such as the Frobenius-norm are possible, and the estimates change only 
by some powers of $h$. For the Frobenius norm, we can for instance employ the estimate 
$\norm{\mathbf{A}}_2 \leq \norm{\mathbf{A}}_F \leq \sqrt{N}\norm{\mathbf{A}}_2$
for $\mathbf{A} \in \mathbb{R}^{N\times N}$. 
\end{remark}

\subsection{$\mathcal{H}$-Cholesky decomposition of ${\mathbf V}$}

$LU$- and Cholesky decompositions are well-es\-tablished tools of numerical linear algebra. 
Properties of these factorizations depend on the choice of the ordering of the unknowns. For 
the $\mathcal{H}$-Cholesky decomposition of Theorem~\ref{th:HLU} below we assume that 
the unknowns are organized in a binary cluster tree ${\mathbb T}_{\mathcal I}$. This induces 
an ordering of the unknowns by requiring that the unknowns of one of the sons are numbered first and
those of the other son later; the precise numbering for the leaves is immaterial for our purposes. 
This induced ordering of the unknowns allows us to speak of {\em block lower triangular} matrices, if the  
block partition $P$ is based on the cluster tree ${\mathbb T}_{\mathcal I}$. With this notation, 
we have the following factorization result: 
\begin{theorem}\label{th:HLU}
Let $\mathbf{V} = \mathbf{C}\mathbf{C}^T$ be the Cholesky decomposition.
Let a partition $P$ of $\mathcal{I}\times \mathcal{I}$ be based on a cluster tree
$\mathbb{T}_{\mathcal{I}}$.
Then, there exist a block lower triangular, blockwise rank-$r$ matrix $\mathbf{C_{\mathcal{H}}}$ such that
\begin{itemize}
\item[(i)] $\displaystyle \frac{\norm{\mathbf{C}-\mathbf{C_{\mathcal{H}}}}_2}{\norm{\mathbf{C}}_2} \leq 
C_{\rm chol}  N^{\frac{3}{2d-2}} {\rm depth}(\mathbb{T}_{\mathcal{I}})
e^{-br^{1/(d+1)}}$
\item[(ii)] 
$\displaystyle\frac{\norm{\mathbf{V}-\mathbf{C_{\mathcal{H}}}\mathbf{C_{\mathcal{H}}}^T}_2}{\norm{\mathbf{V}}_2} 
 \leq 2C_{\rm chol} N^{\frac{3}{2d-2}} {\rm depth}(\mathbb{T}_{\mathcal{I}}) e^{-br^{1/(d+1)}} \\
\phantom{\norm{\mathbf{V}-\mathbf{C_{\mathcal{H}}}\mathbf{C_{\mathcal{H}}}^T}_2  \leq}
+ C_{\rm chol}^2 N^{\frac{3}{d-1}} {\rm depth}(\mathbb{T}_{\mathcal{I}})^2 e^{-2br^{1/(d+1)}}$,
\end{itemize}
where $C_{\rm chol} = C_{\rm sp}C_{\rm sc}\sqrt{\kappa_2(\mathbf{V})}$, 
with the sparsity constant $C_{\rm sp}$ of (\ref{eq:sparsityConstant}), 
the spectral condition number $\kappa_2(\mathbf{V}) := \norm{\mathbf{V}}_2 \norm{\mathbf{V}^{-1}}_2$, 
and a constant $C_{\rm sc}$
depending only on $\Omega,d$, the $\gamma$-shape regularity of the quasiuniform triangulation $\mathcal{T}_h$, 
and $\eta$.
\end{theorem}

\section{Local approximation from low dimensional spaces}
\label{sec:Approximation-solution}
For a given function $f \in L^2(\Gamma)$, we consider the boundary integral equation
\begin{equation*}
V\phi = f \quad \text{on} \; \Gamma.
\end{equation*}
Here, we may consider the simple-layer operator $V \in L(H^{-1}(\Gamma),L^2(\Gamma))$ as a mapping
from $H^{-1}(\Gamma)$ to $L^2(\Gamma)$, see, e.g., \cite{SauterSchwab}.
The discrete variational problem is to find $\phi_h \in S^{0,0}({\mathcal T}_h)$ such that
\begin{equation} \label{eq:model}
\langle V \phi_h,\psi_h\rangle = \langle f,\psi_h\rangle \qquad
\forall \psi_h \in S^{0,0}({\mathcal T}_h).
\end{equation}
With $\mathbf{V}$ from \eqref{eq:Galerkin} and $\mathbf{b} \in \mathbb{R}^N$ defined by $\mathbf{b}_j = \skp{f,\chi_j}$,
the variational problem 
\eqref{eq:model} is equivalent to solving the linear system 
\begin{equation}
\label{eq:discreteSystem}
\mathbf{V} \mathbf{x} = \mathbf{b}. 
\end{equation}
By ellipticity of the simple-layer operator, both problems \eqref{eq:model} and \eqref{eq:discreteSystem} 
have a unique solution. 
The solution $\mathbf{x} \in \mathbb{R}^N$ is linked to \eqref{eq:model} via $\phi_h = \sum_{j=1}^N \mathbf{x}_j \chi_j$. \\

In the following, we repeatedly employ the $L^2(\Gamma)$-orthogonal projection 
$\Pi^{L^2}:L^2(\Gamma)\!\rightarrow\! S^{0,0}(\mathcal{T}_h)$ defined by
\begin{equation}\label{eq:L2projection} 
\skp{\Pi^{L^2} v,\psi_h} = \skp{v,\psi_h} \quad \forall \psi_h \in S^{0,0}(\mathcal{T}_h). 
\end{equation}

The question of approximating 
the matrix block $\mathbf{V}^{-1}|_{\tau\times\sigma} \approx \mathbf{X}_{\tau \sigma}\mathbf{Y}_{\tau \sigma}^T$ 
can be rephrased in terms of functions and function spaces 
as the question of how well $\phi_h|_{B_{R_\tau}}$ can be approximated from low dimensional spaces
for (arbitrary) data $f \in L^2(\Gamma)$ with $\operatorname*{supp} f \subset B_{R_{\sigma}} \cap \Gamma$. 
The present section is devoted to the proof of 
such an approximation result formulated in the following Proposition~\ref{thm:function-approximation}.  

\begin{proposition}\label{thm:function-approximation}
Let $(\tau,\sigma)$ be a cluster pair with bounding boxes $B_{R_\tau}$, $B_{R_\sigma}$. 
Assume $ \eta \operatorname{dist}(B_{R_\tau},B_{R_\sigma}) \geq \operatorname*{diam}(B_{R_\tau})$ for some $\eta > 0$, 
and $R_{\tau} \leq 2 \operatorname*{diam}(\Omega)$. 
Fix $q \in (0,1)$. 
Then, for each $k\in\mathbb{N}$ there exists a subspace
$W_k$ of $S^{0,0}({\mathcal T}_h)$ with $\dim W_k\leq C_{\rm dim} (2+\eta)^d q^{-d}k^{d+1}$ 
such that for arbitrary $f \in L^2(\Gamma)$ with 
$\operatorname*{supp}  f  \subset B_{R_\sigma}\cap\Gamma$, the solution $\phi_h$ of \eqref{eq:model} 
satisfies 
\begin{equation}
\label{eq:thm:function-approximation-1}
\min_{w \in W_k} \|\phi_h - w\|_{L^2(B_{R_{\tau}}\cap \Gamma)} 
\leq C_{\rm box} h^{-2} q^k \|\Pi^{L^2}f\|_{L^2(\Gamma)}
\leq C_{\rm box} h^{-2} q^k \|f\|_{L^2(\Gamma)}.
\end{equation}
The constants $C_{\rm dim}$,  $C_{\rm box}>0$ depend only on $\Omega$, $d$, 
and the $\gamma$-shape regularity of the quasiuniform triangulation $\mathcal{T}_h$.
\end{proposition}

The proof of Proposition~\ref{thm:function-approximation} will be given at the end 
of this section and relies on several observations. First, the potential
\begin{equation*}
u(x) := \widetilde{V}\phi_h(x) = \int_{\Gamma} G(x-y)\phi_h(y)ds_y, \qquad x \in \mathbb{R}^d \setminus \Gamma, 
\end{equation*}
generated by the solution $\phi_h$ of \eqref{eq:model} 
is harmonic on $\Omega$ as well as on $\Omega^c := \mathbb{R}^d \setminus \overline{\Omega}$ and 
satisfies the jump conditions 
\begin{align}\label{eq:jumpconditions}
[\gamma_0 u] &:= \gamma_0^{\text{ext}}u - \gamma_0^{\text{int}}u = 0 \in H^{1/2}(\Gamma), \nonumber \\
[\partial_n u] &:= \gamma_1^{\text{ext}}u - \gamma_1^{\text{int}}u = - \phi_h \in H^{-1/2}(\Gamma).
\end{align}
Here, $\gamma_0^{\text{ext}},\gamma_0^{\text{int}}$ denote the exterior and interior trace operator and 
$\gamma_1^{\text{ext}},\gamma_1^{\text{int}}$ the exterior and interior conormal derivative, 
see, e.g., \cite{SauterSchwab}. Hence, the potential $u$ is in a space of piecewise harmonic functions, 
and the jump of the normal derivative is piecewise constant on the mesh ${\mathcal T}_h$. 
These properties will characterize the spaces ${\mathcal H}_h(D)$ to be introduced below. 
The second observation is an orthogonality condition. For $f$ with $\operatorname*{supp} f  \subset B_{R_{\sigma}}\cap \Gamma$
equation \eqref{eq:model} implies 
\begin{equation}
\label{eq:orthogonality-foo}
\skp{u,\psi_{h}} = \skp{f,\psi_{h}} = 0
\quad \forall \psi_{h} \in S^{0,0}({\mathcal T}_{h})\,\text{with}\, \operatorname*{supp} \psi_{h}  \subset \Gamma \setminus B_{R_{\sigma}}.  
\end{equation}
With the admissibility condition 
$\eta{\rm dist}(B_{R_{\tau}},B_{R_{\sigma}})\!\geq\!\min\{\operatorname*{diam}(B_{R_{\tau}}),
\operatorname*{diam}(B_{R_{\sigma}})\}\! >\! 0$, 
this leads to the orthogonality condition 
\begin{equation}
\label{eq:orthogonality}
\skp{u,\psi_{h}} = 0
\quad \forall \psi_{h} \in S^{0,0}({\mathcal T}_{h}) \,\text{with}\, \operatorname*{supp} \psi_{h}  \subset B_{R_{\tau}} \cap \Gamma,  
\end{equation}
i.e., on $B_{R_\tau} \cap \Gamma$ the potential $u$ is orthogonal 
to piecewise constants. 

With these observations we are able to 
prove a Caccioppoli-type estimate \linebreak (Lemma~\ref{lem:Caccioppoli}) 
for piecewise harmonic functions satisfying the orthogonality \eqref{eq:orthogonality}.
 Then, a low dimensional approximation result (Lemma~\ref{lem:lowdimapp})
derived by Scott-Zhang interpolation of the Galerkin solution $\phi_h$, can be iterated 
as in \cite{BebendorfHackbusch,Boerm},
which finally leads to exponential convergence (Lemma~\ref{cor:lowdimapp}). 

\subsection{Properties of piecewise polynomial spaces}

For an edge/face $T \subset \Gamma$ with affine parametrization $\xi$ and $p \geq 0$,
we let $P_{p}(T)$ be the space 
$P_{p}(T) := \big\{\pi \circ \xi|_T\, : \, \pi\in P_p(\mathbb{R}^{d-1}) \big\}$ 
of polynomials of degree $p$. 
Moreover, we define $S^{p,1}(\mathcal{T}_h):=\big\{{v\in C(\Gamma)}\, : \, v|_T \in P_p(T) \; \forall T \in \mathcal{T}_h\big\}$ 
to be the space of all $\mathcal{T}_h$-piecewise polynomials of degree $p$ that 
are continuous on $\Gamma$.   

Throughout this section we make use of the Scott-Zhang projection 
\begin{equation*} 
I_h: H^1(\Gamma) \rightarrow S^{1,1}({\mathcal T}_h) 
\end{equation*}
introduced in \cite{ScottZhang}.  
We note that the Scott-Zhang projection in \cite{ScottZhang} is only defined for functions in $H^1(\Gamma)$, 
but by averaging only over triangles (and not over faces) it may also be well defined for functions in $L^2(\Gamma)$. 

By
\begin{equation*} 
\omega_T := \text{interior}\left(\bigcup\left\{\overline{T'} \in \mathcal{T}_h \;:\; 
T \cap T' \neq \emptyset\right\}\right), 
\end{equation*}
we denote the element patch of $T$, 
which contains $T$ and all elements $T' \in \mathcal{T}_h$ that share a node with $T$. 
The operator $I_h$ has well-known 
local approximation properties for $H^{\ell}$-functions
\begin{equation}\label{eq:SZapprox}
\norm{v-I_h v}_{H^m(T)}^2 \leq C h^{2(\ell-m)}\abs{v}_{H^{\ell}(\omega_T)}^2, 
\quad 0 \leq m \leq 1, \ m \leq \ell \leq 1.
\end{equation}
The constant $C>0$ depends only on the $\gamma$-shape regularity of $\mathcal{T}_h$ and the dimension $d$. 
With the triangle inequality this also implies the $L^2$-stability estimate
\begin{equation}\label{eq:L2SZ}
 \norm{I_h v}_{L^2(T)}\leq C\norm{v}_{L^2(\omega_T)}.
\end{equation}

The following lemma constructs a stable operator 
on $L^2(\Gamma)$ that features additional orthogonality properties: 

\begin{lemma}\label{lem:interpolator}
There exists a linear operator
$\mathcal{J}_h : L^2(\Gamma) \rightarrow S^{d,1}(\mathcal{T}_h)$ 
with the following properties for all $v \in H^1(\Gamma)$
\begin{enumerate}
\def\theenumi{\roman{enumi}}
 \item
\label{item:lem:interpolator-i}
$\norm{\mathcal{J}_h v}_{L^2(T)} \leq C \norm{v}_{L^2(\omega_T)} \quad \forall T \in \mathcal{T}_h$;
\item
\label{item:lem:interpolator-ii}
$\norm{\nabla \mathcal{J}_h v}_{L^2(T)} \leq C \norm{\nabla v}_{L^2(\omega_T)} \quad \forall T \in \mathcal{T}_h$;
\item
\label{item:lem:interpolator-iii}
$\skp{v- \mathcal{J}_h v,\psi} = 0 \quad \forall \psi \in S^{0,0}(\mathcal{T}_h)$;
\item
\label{item:lem:interpolator-iv}
$\norm{v- \mathcal{J}_h v}_{L^2(T)} \leq C h \norm{\nabla v}_{L^2(\omega_T)} \quad \forall T \in \mathcal{T}_h$.
\end{enumerate}
The constant $C>0$ depends only on $d$ and the $\gamma$-shape regularity 
of the quasiuniform triangulation $\mathcal{T}_h$.
\end{lemma} 

\begin{proof}
Let $b_T \in S^{d,1}(\mathcal{T}_h)$ be the element bubble function for each $T \in \mathcal{T}_h$,
which is the product of the $d$ hat-functions associated with $T$ and 
scaled such that 
$\norm{b_T}_{\infty} = 1$. Denote by $\chi_T$ the characteristic function 
of $T$. With the Scott-Zhang projection $I_h$, we define
\begin{equation*}
\mathcal{J}_h v := I_h v + \sum_{T \in \mathcal{T}_h} b_T \frac{\skp{v-I_h v, \chi_T}}{\int_T b_T}.
\end{equation*} 
For $T' \in \mathcal{T}_h$ we have
\begin{align*}
\skp{v - \mathcal{J}_h v, \chi_{T'}} &= \skp{v - I_h v - \sum_{T \in \mathcal{T}_h} b_T 
\frac{\skp{v-I_h v, \chi_T}}{\int_T b_T}, \chi_{T'}} \\
&= \skp{v- I_h v, \chi_{T'}} - \frac{\skp{b_{T'},\chi_{T'}}}{\int_{T'} b_{T'}}\skp{v-I_h v, \chi_{T'}} = 0,
\end{align*}
which proves (\ref{item:lem:interpolator-iii}).
The Cauchy-Schwarz inequality and the approximation property \eqref{eq:SZapprox} of $I_h$ imply
\begin{align*}
\norm{v-\mathcal{J}_h v}_{L^2(T)} &\leq \norm{v-I_h v}_{L^2(T)} + \frac{\norm{b_T}_{L^2(T)}}{\abs{\int_T b_T}}\abs{T}^{1/2} 
\norm{v-I_h v}_{L^2(T)} \\
&\lesssim \norm{v-I_h v}_{L^2(T)} \lesssim h \norm{\nabla v}_{L^2(\omega_T)}.
\end{align*}
This proves (\ref{item:lem:interpolator-iv}). 
The first assertion (\ref{item:lem:interpolator-i}) follows with the same argument due to the $L^2$-stability
of the Scott-Zhang projection $I_h$ \eqref{eq:L2SZ}.
Finally, we get
\begin{align*}
\norm{\nabla(v- \mathcal{J}_h v)}_{L^2(T)} &\leq \norm{\nabla (v-I_h v)}_{L^2(T)} + 
\frac{\norm{\nabla b_T}_{L^2(T)}}{\abs{\int_T b_T}}\abs{T}^{1/2} \norm{v-I_h v}_{L^2(T)} \\
&\lesssim \norm{\nabla(v-I_h v)}_{L^2(T)} + \frac{1}{h}\norm{v-I_h v}_{L^2(T)} \lesssim \norm{\nabla v}_{L^2(\omega_T)},
\end{align*}
and the triangle inequality finishes the proof of (\ref{item:lem:interpolator-ii}).
\end{proof}

The following inverse inequalities also holds for locally refined $K$-meshes, 
but we will only require it for the quasiuniform mesh $\mathcal{T}_h$ at hand.

\begin{lemma}[{\cite[Thm~{4.1}, Thm.~{4.7}]{GHS}}]
\label{lem:inverseinequality}
There is a constant $C>0$ depending only on $\Gamma$, the $\gamma$-shape regularity
of the quasiuniform triangulation ${\mathcal T}_h$, and the polynomial degree $p$ such that 
\begin{align}
\label{eq:inverse}
\norm{v_h}_{H^{1/2}(\Gamma)}&\leq  C h^{-1/2}\norm{v_h}_{L^2(\Gamma)} 
\qquad \forall v_h \in S^{p,1}({\mathcal T}_h),\\ 
\label{eq:inverse-a}
\norm{v_h}_{L^{2}(\Gamma)}&\leq  C h^{-1/2}\norm{v_h}_{H^{-1/2}(\Gamma)} 
\qquad \forall v_h \in S^{p,0}({\mathcal T}_h). 
\end{align}
\end{lemma}

\subsection{The spaces ${\mathcal H}_h(D)$ and ${\mathcal H}_{h,0}(D,\Gamma_{\rho})$ 
of piecewise harmonic functions}

For an index set $\rho \subset \mathcal{I}$, 
let $\Gamma_{\rho} \subset \Gamma$ be a relative open polygonal manifold, consisting of the union of elements
in $\mathcal{T}_h$ associated with the elements in $\rho$, i.e.,
\begin{equation}\label{eq:Gammarho}
\Gamma_{\rho} = {\rm interior}\left(\bigcup_{j\in\rho} \overline{T_j} \right).
\end{equation}
Let $D$ be a domain and set $D^- := D\cap \Omega$ and 
$D^+ := D\cap \overline{\Omega}^c$. 
A function $v \in H^1(D^+ \cup D^-)$ is called piecewise harmonic, if 
\begin{equation*}
\int_{D^{\pm}}{\nabla v \cdot \nabla \varphi}
\,dx = 0 \quad \forall  \varphi \in C_0^{\infty}(D^{\pm}).
\end{equation*}

\begin{remark}
For a piecewise harmonic function $v \in H^1(D^+\cup D^-)$, 
we can define the jump of the normal derivative $[\partial_{n}v]|_{D\cap\Gamma}$ on $D\cap\Gamma$ as the functional
\begin{equation}\label{eq:unjumpdef}
\skp{[\partial_{n}v]|_{D\cap\Gamma},\varphi} := \int_{D^+\cup D^-}\nabla v \cdot \nabla \varphi dx \quad \forall \varphi \in H^1_0(D). 
\end{equation}
We note that the value $\skp{[\partial_{n}v]|_{D\cap\Gamma},\varphi}$ depends only on $\varphi|_{D\cap \Gamma}$ 
in the sense that \linebreak
$\skp{[\partial_{n}v]|_{D\cap\Gamma},\varphi} = 0$ for all 
$\varphi \in C_0^{\infty}(D)$ with $\varphi|_{D\cap\Gamma} = 0$.
Moreover, if $[\partial_{n}v]|_{D\cap\Gamma}$ is a function in $L^2(D\cap\Gamma)$, it is unique.
The definition \eqref{eq:unjumpdef} is consistent with \eqref{eq:jumpconditions} in the following sense:
For the potential $V\phi_h$ with $\phi_h \in S^{0,0}(\mathcal{T}_h)$, we have the jump condition
$[\partial_n V \phi_h]|_{D\cap\Gamma} = -\phi_h|_{D\cap\Gamma}$.
\end{remark}

The space of piecewise harmonic functions on $D$ with piecewise constant jump of the normal derivative is defined by
\begin{align*}
\mathcal{H}_h(D) &:= \{v\in H^1(D^+\cup D^-) \colon \text{ $v$ is piecewise harmonic}, \\
 &  \quad \phantom{v\in H^1(D^+\cup D^-): }\exists \widetilde{v} \in S^{0,0}({\mathcal T}_h) \;
\mbox{s.t.} \ 
[\partial_{n}v]|_{D \cap \Gamma} = \widetilde{v}|_{D \cap \Gamma}\}.
\end{align*}
 
The potential $u = \widetilde{V}\phi_h$ for the problem \eqref{eq:model}
indeed satisfies $u \in \mathcal{H}_h(D) \cap H^1(D)$ for any bounded domain $D$. 
Moreover, for a bounding box $B_{R_{\sigma}}$ with \eqref{eq:admissibility}, the potential $u$ 
additionally satisfies the orthogonality condition \eqref{eq:orthogonality}. These observations are captured 
by the following space ${\mathcal H}_{h,0}(D,\Gamma_{\rho})$: 
\begin{align}\label{eq:ApproxSpaceOrtho}
{\mathcal H}_{h,0}(D,\Gamma_{\rho})&:={\mathcal H}_h(D) \cap 
\{v \in H^1(D)\colon \operatorname*{supp} [\partial_{n}v]|_{D\cap\Gamma} \subset \overline{\Gamma_{\rho}},\nonumber\\
&  \phantom{{\mathcal H}(D) \cap \{v \in } \langle v,\varphi\rangle
=0\; \forall \varphi \in S^{0,0}({\mathcal T}_h)\,\text{with}\, \operatorname*{supp} \varphi \subset D\cap 
\overline{\Gamma_{\rho}}\}.
\end{align}
For the proof of Proposition~\ref{thm:function-approximation} and subsequently of Theorem~\ref{th:blockapprox}
and Theorem~\ref{th:Happrox}, we will only need the case $\Gamma_{\rho} = \Gamma$. 
The general case of the screen problem $\Gamma_{\rho} \subsetneq \Gamma$ will only be required for our analysis 
of the $\mathcal{H}$-Cholesky decomposition in Section~\ref{sec:LU-decomposition}.

The following lemma shows that this space is a closed subspace of $H^1(D\setminus\Gamma)$; 
later, this property will allow us to consider the
orthogonal projection from $H^1(D\backslash\Gamma)$ onto $\mathcal{H}_h(D)$ and from
$H^1(D)$ onto ${\mathcal H}_{h,0}(D,\Gamma_{\rho})$. 

\begin{lemma}\label{lem:closedsubspace}
The space $\mathcal{H}_h(D)$ is a closed subspace of $H^1(D\setminus\Gamma)$, 
and $\mathcal{H}_{h,0}(D,\Gamma_{\rho})$ is a closed subspace of $H^1(D)$. 
\end{lemma}
\begin{proof}
We first show that $\mathcal{H}_h(D)$ is a closed subspace of $H^1(D\setminus \Gamma)$. 
Let $(v^j)_{j\in \mathbb{N}} \subset \mathcal{H}_h(D)$ be a sequence converging to $v \in H^1(D\backslash \Gamma)$. 
For $\varphi \in C_0^{\infty}(D^{\pm})$, we have
\begin{equation*}
\skp{\nabla v, \nabla \varphi}_{L^2(D^{\pm})} = \lim_{j\rightarrow \infty} \skp{\nabla v^j, \nabla \varphi}_{L^2(D^{\pm})} = 0.
\end{equation*}
Hence, $v$ is piecewise harmonic on $D^+\cup D^-$. 

Pick $\varphi \in C_0^{\infty}(D)\setminus \{0\}$ with $\operatorname*{supp} \varphi  \cap \Gamma \subset T \in \mathcal{T}_h$. 
Then 
\begin{align*}
[\partial_{n}v^j]|_T\skp{1,\varphi}_{L^2(T)} = \skp{[\partial_{n}v^j], \varphi}_{L^2(T)} =
 \skp{\nabla v^j, \nabla \varphi}_{L^2(D\backslash \Gamma)} \stackrel{j \rightarrow \infty}{\longrightarrow} 
\skp{\nabla v,\nabla \varphi}_{L^2(D\backslash \Gamma)}
\end{align*}
shows that the 
piecewise constant function $[\partial_{n}v^j]$ converges elementwise. 
Hence, the sequence $\left([\partial_{n}v^j]\right)_{j \in \mathbb{N}}$ converges pointwise 
to a piecewise constant function $\widetilde{v}$. This piecewise constant limit $\tilde v$ coincides with 
the jump of the normal derivative $[\partial_{n}v] \in H^{-1/2}(\Gamma)$ as the following calculation for arbitrary  
$\varphi \in C_0^{\infty}(D)$ shows: 
\begin{align*}
\skp{\widetilde{v},\varphi}_{L^2(D\cap \Gamma)} &= \lim_{j\rightarrow\infty}\skp{[\partial_{n}v^j], \varphi}_{L^2(D\cap \Gamma)} 
 = \skp{\nabla v,\nabla \varphi}_{L^2(D\backslash \Gamma)} = \skp{[\partial_{n}v],\varphi}_{L^2(D \cap \Gamma)}.
\end{align*}
Finally, $\mathcal{H}_{h,0}(D,\Gamma_{\rho})$ is a closed subspace of $H^1(D)$, since 
$\mathcal{H}_{h}(D)$ is a closed subspace of $H^1(D\setminus\Gamma)$ and the intersection
of closed spaces is again closed. 
\end{proof}

We will derive an approximation of the Galerkin solution $\phi_h$ by approximating the potential $u = \widetilde V \phi_h$. 
In view of the relation $\phi_h = - [\partial_n u]$ we have to control the jump of the normal derivative 
by a norm of $u$.
Lemma~\ref{lem:estimateunjump} below provides such an estimate,
 which may be seen as an inverse estimate, since $[\partial_{n}u]$ is 
a discrete function. For its proof, we need the following 
Lemma~\ref{lem:quasi-1D} as well as the definition of ``concentric boxes''. 

\begin{definition}
Two (open) boxes $B_R$, $B_{R^\prime}$
are said to be concentric boxes with side lengths $R$ and $R^\prime$, if 
they have the same barycenter and $B_R$ can be obtained by a stretching
of $B_{R^\prime}$ by the factor $R/R^\prime$ taking their common barycenter
as the origin. 
\end{definition}

The following Lemma~\ref{lem:quasi-1D}  is quite classical, and we 
include its short proof for the reader's convenience.  
\begin{lemma}
\label{lem:quasi-1D}
\begin{enumerate}
\def\theenumi{\roman{enumi}}
\item
\label{item:lem:quasi-1D-i}
For $R \leq 1$ denote by $S_{R}:=\{x \in \mathbb{R}^d\,:\, \operatorname{dist}(x,\Gamma) < R\}$
the tubular neighborhood of $\Gamma$ of width $R$. Then, there is a constant
$C > 0$ that depends only on $\Gamma$ such that 
\begin{equation*}
\|v\|_{L^2(S_{R})} \leq C \left[ \sqrt{R} \|\gamma_0^{\rm int} v\|_{L^2(\Gamma)} + 
R \|\nabla v\|_{L^2(S_{R})}\right]
\qquad \forall v \in H^1(S_{R}).
\end{equation*} 
\item
\label{item:lem:quasi-1D-ii}
Let $\delta$, $R > 0$. Let $B_R$ and $B_{(1+\delta)R}$ be two 
concentric boxes with side lengths  $R$ and $(1+\delta)R$. 
Then, there is a constant $C > 0$, which depends only on the dimension $d$, such that 
for all $v \in H^1(B_{(1+\delta)R})$, we have
\begin{equation*}
\|v\|^2_{L^2(B_{(1+\delta)R}\setminus B_R)}  \leq 
 C \delta R 
\left(
\frac{1}{(1+\delta)R} \|v\|^2_{L^2(B_{(1+\delta)R})}
+ (1+\delta)R \|\nabla v\|^2_{L^2(B_{(1+\delta)R})}
\right). 
\end{equation*}
\end{enumerate}
\end{lemma}
\begin{proof}
{\sl ad (\ref{item:lem:quasi-1D-i}):} 
For 
smooth univariate functions $v$ the fundamental theorem of calculus 
yields $v(x) = v(0) + \int_0^x v'(t)\,dt$. 
Hence, the Young inequality and the Cauchy-Schwarz inequality yield
$v^2(x) \leq 2v^2(0) +2 R \|v^\prime\|_{L^2(0,R)}^2$. 
Integration over the interval $(0,R)$ gives 
\begin{equation*}
\|v\|^2_{L^2(0,R)} \leq 2R v^2(0) + 2 R^2 \|v^\prime\|^2_{L^2(0,R)}.
\end{equation*}
This 1D result implies the desired estimate by using (locally) boundary 
fitted coordinates. 

{\sl ad (\ref{item:lem:quasi-1D-ii}):} We start with the 1D Gagliardo-Nirenberg
inequality for the interval $I = (0,1)$: 
$\|v\|^2_{L^\infty(I)} \leq C \|v\|_{L^2(I)} \|v\|_{H^1(I)} 
\leq C \|v\|^2_{L^2(I)} + C \|v\|_{L^2(I)} \|v^\prime\|_{L^2(I)}$. 
A scaling argument then yields 
\begin{align*}
\|v\|^2_{L^\infty(0,(1+\delta)R)} 
&\lesssim  \frac{1}{(1+\delta)R} \|v\|^2_{L^2(0,(1+\delta)R)} 
+ \|v\|_{L^2(0,(1+\delta)R)} \|v^\prime\|_{L^2(0,(1+\delta)R)} 
\\
&\lesssim    \frac{1}{(1+\delta)R} \|v\|^2_{L^2(0,(1+\delta)R)} 
+ (1+\delta)R \|v^\prime\|^2_{L^2(0,(1+\delta)R)}.
\end{align*}
We may assume that $B_{(1+\delta)R} = (0,(1+\delta)R)^d$. Then, 
this 1D estimate implies 
\begin{align*}
\|v\|^2_{L^2((0,\delta R) \times (0,(1+\delta)R)^{d-1})}
&\lesssim 
\frac{\delta R }{(1+\delta)R} \|v\|^2_{L^2(B_{(1+\delta)R})} \\
&\phantom{\leq} \, + \delta R (1+\delta)R \|\nabla v\|^2_{L^2(B_{(1+\delta)R})}. 
\end{align*}
By arguing similarly for the remaining parts of 
$B_{(1+\delta)R}\setminus B_R$, we get the desired result.
\end{proof}

\begin{lemma}\label{lem:estimateunjump}
Let $\delta \in (0,1)$, $R > 0$ be such that 
$\frac{h}{R}\leq \frac{\delta}{4}$. 
Let $B_R$, $B_{(1+\delta)R}$ be two concentric boxes of 
side lengths $R$ and $(1+\delta)R$. 
Then, there exists a constant $C> 0$ depending only on $\Omega$, $d$, and the $\gamma$-shape regularity
of the quasiuniform triangulation $\mathcal{T}_h$, such that for 
all $v \in \mathcal{H}_h(B_{(1+\delta)R})$
\begin{align}
\label{eq:lem:estimateunjump-ii}
\norm{[\partial_{n}v]}_{L^2(B_{(1+\delta/2)R}\cap \Gamma)} 
\leq C h^{-1/2} 
\norm{\nabla v}_{L^2(B_{(1+\delta)R})}.
\end{align}
\end{lemma}

\begin{proof} 
We prove \eqref{eq:lem:estimateunjump-ii} in two steps, the first step being the proof
of the auxiliary estimate \eqref{eq:lem:estimateunjump-i} below. 
The second step shows \eqref{eq:lem:estimateunjump-ii} with the aid
of \eqref{eq:lem:estimateunjump-i} and a simple covering argument. 

{\bf Step 1:} We show the following assertion: 
If $\frac{h}{r} \leq \frac{\varepsilon}{4}$ for $r,\varepsilon >0$, then 
there exists a constant $C>0$ depending only on the shape regularity
constant $\gamma$, the domain $\Omega$, and $d$ such that
for all $v \in \mathcal{H}_h(B_{(1+\varepsilon)r})$ 
\begin{align}
\label{eq:lem:estimateunjump-i}
\norm{[\partial_{n}v]}_{L^2(B_{r}\cap \Gamma)} \leq  C h^{-1/2} 
\sqrt{1 + \frac{1}{\varepsilon}}
\norm{\nabla v}_{L^2(B_{(1+\varepsilon)r})}. 
\end{align}
To see this, let $\mathcal{E}^{\rm{int}}:H^{1/2}(\Gamma)\rightarrow H^1(\Omega)$ 
and $\mathcal{E}^{\rm{ext}}:H^{1/2}(\Gamma) \rightarrow H^1(\overline{\Omega}^c)$ be (bounded, linear) lifting operators
for $\Omega$ and $\Omega^c$ (cf. \cite[Thm.~{5.7}]{Necas}). Then, introduce the
(bounded, linear) lifting $\mathcal{L}:H^{1/2}(\Gamma)\rightarrow H^1(\mathbb{R}^d)$
by 
\begin{equation*}
\mathcal{L} w := \left\{
\begin{array}{l}
 \mathcal{E}^{\rm{int}}  w \quad\textrm{in }\overline{\Omega}, \\
 \mathcal{E}^{\rm{ext}}  w \quad \textrm{in }\overline{\Omega}^c. 
 \end{array}
 \right.
\end{equation*}

Lemma~\ref{lem:interpolator} provides an operator 
$\mathcal{J}_h:L^2(\Gamma)\rightarrow S^{d,1}(\mathcal{T}_h) \subset H^1(\Gamma)$. Furthermore, 
$w- \mathcal{J}_h w$ is orthogonal to piecewise constant functions so that 
\begin{equation}
\label{eq:temp10}
\norm{[\partial_{n}v]}_{L^2(B_{r}\cap \Gamma)} = \sup_{\substack{w \in L^2(\Gamma) \\ \operatorname*{supp} w \subset B_r}} 
\frac{\skp{[\partial_{n}v],w}}{\norm{w}_{L^2(\Gamma)}} 
= \sup_{\substack{w \in L^2(\Gamma)\\ \operatorname*{supp} w \subset B_r}} \frac{\skp{[\partial_{n}v],\mathcal{J}_h w}}{\norm{w}_{L^2(\Gamma)}}.
\end{equation}
Note that the construction of $\mathcal{J}_h$ implies 
$\operatorname*{supp} \mathcal{J}_h w \subset B_{(1+\varepsilon/2)r} \cap \Gamma$. 
Let $\eta$ be a smooth cut-off function with $0\leq\eta\leq 1$, $\eta \equiv 1$ on 
$B_{(1+\varepsilon/2)r}$, and $\operatorname*{supp} \eta   \subset B_{(1+\varepsilon)r}$ and 
$\norm{\nabla \eta}_{L^{\infty}(B_{(1+\varepsilon)r})} \lesssim \frac{1}{\varepsilon r}$. 
With the lifting $\mathcal{L} (\mathcal{J}_h w)$, we can estimate 
\begin{align}\label{eq:temp1}
\skp{[\partial_{n}v],\mathcal{J}_h w}_{L^2(B_{r+2h}\cap \Gamma)} &=
\skp{[\partial_{n}v],\gamma_0^{\text{int}}\eta\mathcal{L}\mathcal{J}_h w}_{L^2(B_{r+2h}\cap \Gamma)} \nonumber \\
&= \int_{B_{(1+\varepsilon)r}}{\nabla v \cdot \nabla(\eta \mathcal{L}\mathcal{J}_h w )dx} \nonumber \\
&\leq \norm{\nabla v}_{L^2(B_{(1+\varepsilon)r})}\norm{\nabla (\eta \mathcal{L}\mathcal{J}_h w)}_{L^2(B_{(1+\varepsilon)r})}.
\end{align}
We have to estimate $\eta \mathcal{L}\mathcal{J}_h w$ further. 
Noting that $\eta \equiv 1$ on $B_r$, we use the product rule to estimate
\begin{align}
\label{eq:temp1a}
\|\nabla (\eta \mathcal{L}\mathcal{J}_h w)\|_{L^2(B_{(1+\varepsilon)r})}
&\lesssim  \frac{1}{\varepsilon r} \|\mathcal{L}\mathcal{J}_h w\|_{L^2(B_{(1+\varepsilon) r}\setminus B_r)}
+ \|\nabla \mathcal{L}\mathcal{J}_h w\|_{L^2(B_{(1+\varepsilon)r})}.
\end{align}
The continuity of the lifting $\mathcal{L}:H^{1/2}(\Gamma) \rightarrow H^1(\mathbb{R}^d)$ and
the inverse estimate \eqref{eq:inverse} of Lemma~\ref{lem:inverseinequality} give 
\begin{align}
\label{eq:temp1aa}
\|\nabla\mathcal{L}\mathcal{J}_h w\|_{L^2(S_{\sqrt{d}(1+\varepsilon)r})} 
&\leq \|\nabla\mathcal{L}\mathcal{J}_h w\|_{L^2(\mathbb{R}^d)}\nonumber\\
&\lesssim  \|\mathcal{J}_h w\|_{H^{1/2}(\Gamma)}
\lesssim \frac{1}{\sqrt{h}} \|\mathcal{J}_h w\|_{L^2(\Gamma)}, 
\end{align}
which is the key step for the treatment of the second term in (\ref{eq:temp1a}). Let us now
turn to the first term in (\ref{eq:temp1a}). 
Using Lemma~\ref{lem:quasi-1D}, (\ref{item:lem:quasi-1D-ii}) and then
Lemma~\ref{lem:quasi-1D}, (\ref{item:lem:quasi-1D-i}) with the observation 
$B_{(1+\varepsilon)r} \subset S_{\sqrt{d}(1+\varepsilon)r}$, we get 
\begin{align*}
\frac{1}{\varepsilon r} \|\mathcal{L}\mathcal{J}_h w\|_{L^2(B_{(1+\varepsilon) r}\setminus B_r)}
&\lesssim \frac{1}{\sqrt{\varepsilon r}\sqrt{(1+\varepsilon)r}} \|\mathcal{L}\mathcal{J}_h w\|_{L^2(B_{(1+\varepsilon)r})} \\
 &\phantom{\leq}\, + 
\frac{\sqrt{(1+\varepsilon)r}}{\sqrt{\varepsilon r}} \|\nabla \mathcal{L}\mathcal{J}_h w\|_{L^2(B_{(1+\varepsilon)r})}
 \\
&\lesssim \frac{1}{\sqrt{\varepsilon r}} 
\|\mathcal{L}\mathcal{J}_h w\|_{L^2(\Gamma)} + \frac{\sqrt{(1+\varepsilon)r}}{\sqrt{\varepsilon r}} 
\|\nabla \mathcal{L}\mathcal{J}_h w\|_{L^2(S_{\sqrt{d}(1+\varepsilon)r})}.  
\end{align*}
We have $\mathcal{L} \mathcal{J}_h w|_{\Gamma} = \mathcal{J}_h w$, and
(\ref{eq:temp1aa}) leads to
\begin{align}
\label{eq:temp2a}
\frac{1}{\varepsilon r} \|\mathcal{L}\mathcal{J}_h w\|_{L^2(B_{(1+\varepsilon) r}\setminus B_r)}
& \lesssim \frac{1}{\sqrt{\varepsilon r}} 
\|\mathcal{J}_h w\|_{L^2(\Gamma)} + h^{-1/2} \sqrt{1+\varepsilon^{-1}} \|\mathcal{J}_h w\|_{L^{2}(\Gamma)}.  
\end{align} 
We note that $\varepsilon r \ge 4h$ so that $1/\sqrt{\varepsilon r} \lesssim h^{-1/2}$. 
Inserting (\ref{eq:temp1aa}) and (\ref{eq:temp2a}) in (\ref{eq:temp1a})
and using 
the $L^2(\Gamma)$-stability of $\mathcal{J}_h$ given by
Lemma~\ref{lem:interpolator}, we obtain 
\begin{align*}
\|\nabla (\eta \mathcal{L}\mathcal{J}_h w)\|_{L^2(B_{(1+\varepsilon)r})} 
\lesssim h^{-1/2} \sqrt{1+\varepsilon^{-1}} \|\mathcal{J}_h w\|_{L^2(\Gamma)}
\lesssim h^{-1/2} \sqrt{1+\varepsilon^{-1}} \|w\|_{L^2(\Gamma)}. 
\end{align*}
Finally, inserting this bound into (\ref{eq:temp1}) and then into 
(\ref{eq:temp10}) allows us to conclude the proof of 
(\ref{eq:lem:estimateunjump-i}). 

{\bf Step 2:}
The bound 
(\ref{eq:lem:estimateunjump-ii}) is shown with the aid of 
(\ref{eq:lem:estimateunjump-i}) and a covering argument. 
We may assume that $B_R = (0,R)^d$. Set $r = \delta R$. 
Let $n \in \mathbb{N}$ be given by $n = \lceil R/r\rceil$. 
Let $x_i$, $i=1,\ldots,(n+1)^d=:N$ be the points of a regular 
grid in the closed box $\overline{B_R}$ with spacing $R/n$. 
For $i=1,\ldots,N$ consider the boxes $B_i:= x_i + (-r/2,r/2)^d$
as well as the scaled boxes 
$\widehat B_i:= x_i + (-r,r)^d$, $i = 1,\ldots,N$. 
The essential properties of these boxes are: 
first, the boxes $B_i$, $i =1,\ldots,N$ cover 
$B_{(1+\delta/2)R}$; secondly, the scaled boxes
$\widehat B_i$, $i=1,\ldots,N$ are contained in $B_{(1+\delta)R}$; thirdly, 
and most importantly, they have a finite overlap property
(with an overlap constant that depends solely on the spatial dimension $d$,
since the ratio of the spacing $R/n$ and the side length $r$ 
satisfies $r/(R/n) = (1/\delta)/ \lceil 1/\delta \rceil 
\in [1/2,1]$ for the case $\delta \in (0,1)$ under consideration here). 
Observing 
$\frac{h}{r} = \frac{h}{\delta R} \leq \frac{1}{4}$ due to our 
assumption $\frac{h}{R} \leq \frac{\delta}{4}$, the
estimate (\ref{eq:lem:estimateunjump-i}) implies for each $i$
\begin{equation*}
\|[\partial_{n}v]\|_{L^2(B_{i}\cap \Gamma)} 
\leq C h^{-1/2} 
\|\nabla v\|_{L^2(\widehat B_{i})}  . 
\end{equation*}
The desired estimate \eqref{eq:lem:estimateunjump-ii} 
follows from the covering and overlap properties. 
\end{proof}

For a box $B_R$ with side length $R$, we introduce the norm
\begin{equation*}
\triplenorm{v}_{h,R}^2 := \left(\frac{h}{R}\right)^2 \norm{\nabla v}^2_{L^2(B_R)} + 
\frac{1}{R^2}\norm{v}_{L^2(B_R)}^2,
\end{equation*}
which is, for fixed $h$, equivalent to the $H^1$-norm. 

Similarly as in \cite{BebendorfHackbusch,Boerm}, a main part of the proof is a Caccioppoli-type inequality, 
which is, for functions in $\mathcal{H}_{h,0}(B_{(1+\delta)R},\Gamma_{\rho})$, stated in the following lemma.

\begin{lemma}\label{lem:Caccioppoli}
Let $\delta \in (0,1)$ and 
$\frac{h}{R} \leq \frac{\delta}{16}$ and let $\Gamma_{\rho} \subset \Gamma$ be of the form \eqref{eq:Gammarho}.
Let $B_R$, $B_{(1+\delta)R}$ be two concentric boxes.
Then, there exists a constant $C > 0$ depending only on 
$\Omega$, $d$, and the $\gamma$-shape regularity
of the quasiuniform triangulation $\mathcal{T}_h$, such that 
for $v \in \mathcal{H}_{h,0}(B_{(1+\delta)R},\Gamma_{\rho})$
\begin{equation}\label{eq:caccioppoli}
\norm{\nabla v}_{L^2(B_{R})} \leq C\frac{1+\delta}{\delta}\triplenorm{v}_{h,(1+\delta)R}.
\end{equation}
\end{lemma}

\begin{proof}
The proof of \eqref{eq:caccioppoli} is done in two steps.\\
{\bf Step 1:}
We show that for $\varepsilon>0$ with $\frac{h}{R}\leq \frac{\varepsilon}{8}$, the estimate
\begin{equation}\label{eq:caccioppoli1}
\|\nabla v\|^2_{L^2(B_{R})} 
\lesssim \frac{h}{\varepsilon R} \|\nabla v\|^2_{L^2(B_{(1+\varepsilon)R})} 
+ \frac{1}{(\varepsilon R)^2} \|v\|^2_{L^2(B_{(1+\varepsilon)R})}. 
\end{equation}
holds.
To see this, let $\eta$ be a smooth cut-off function with 
$\operatorname*{supp} \eta \subset B_{(1+\varepsilon/4)R}$ and $\eta \equiv 1$ on $B_{R}$, and 
$\norm{\nabla \eta}_{L^{\infty}(B_{(1+\varepsilon)R})}\lesssim \frac{1}{\varepsilon R}$. 
We will need a second smooth cut-off function $\widetilde \eta$ 
with $\operatorname*{supp} \widetilde \eta \subset B_{(1+\varepsilon)R}$ and 
$\widetilde \eta \equiv 1$ on $B_{(1+\varepsilon/2)R}$ and 
$\|\nabla \widetilde \eta\|_{L^\infty(B_{(1+\varepsilon)R})} \lesssim \frac{1}{\varepsilon R}$. 
Since $h$ is the maximal element diameter, $8h \leq \varepsilon R$ implies $T \subset B_{(1+\varepsilon/2)R}$ 
for all $T \in \mathcal{T}_h$ with $T \cap \operatorname*{supp} \eta  \neq \emptyset$.
Integration by parts, the fact that $v$ is piecewise harmonic and $\operatorname*{supp} ([\partial_{n}v] |_{B_{(1+\varepsilon)R} \cap \Gamma})
\subset \overline{\Gamma_{\rho}}$ lead to 
\begin{align}\label{eq:temp3}
\norm{\nabla(\eta v)}_{L^2(B_{(1+\varepsilon)R})}^2 &= 
\int_{B_{(1+\varepsilon)R}}\nabla(\eta v) \cdot \nabla(\eta v) \, dx \nonumber \\
 &= \int_{B_{(1+\varepsilon)R}}\nabla v \cdot \nabla(\eta^2 v) +  
v^2 \abs{\nabla\eta}^2  dx \nonumber \\
&= \int_{\Gamma\cap B_{(1+\varepsilon)R}}{\eta^2 [\partial_{n}v] v \, ds_x} + 
\int_{B_{(1+\varepsilon)R}}{v^2\abs{\nabla \eta}^2  dx} \nonumber \\
&= \int_{\Gamma}{\eta^2 [\partial_{n}v] v \, ds_x} +
\int_{B_{(1+\varepsilon)R}}{v^2\abs{\nabla \eta}^2  dx}, 
\end{align}
where in the last step we used the support property 
$\operatorname*{supp} \eta \subset B_{(1+\varepsilon/4)R}$ to extend 
the function $\eta^2 [\partial_{n}v] v$, which is defined on $\Gamma \cap B_{(1+\varepsilon)R}$, 
by zero to the whole set $\Gamma$. 
We first focus on the surface integral in \eqref{eq:temp3}. 
With the $L^2(\Gamma)$-orthogonal projection $\Pi^{L^2}$ onto $S^{0,0}(\mathcal{T}_h)$ from \eqref{eq:L2projection}, 
we get by definition of the space $\mathcal{H}_{h,0}(B_{(1+\varepsilon)R},\Gamma_{\rho})$
that $\operatorname*{supp} \Pi^{L^2}(\eta^2 [\partial_{n}v]) \subset \overline{\Gamma_{\rho}} \cap B_{(1+\varepsilon)R}$. Therefore,
we can use the orthogonality \eqref{eq:ApproxSpaceOrtho} satisfied by $v$ to get 
\begin{align}
\nonumber 
\langle \eta^2 [\partial_{n}v],v\rangle &= 
\langle \eta^2 [\partial_{n}v] - \Pi^{L^2} (\eta^2 [\partial_{n}v]),v\rangle \\
\nonumber &= 
\langle \eta^2 [\partial_{n}v] - \Pi^{L^2} (\eta^2 [\partial_{n}v]),\widetilde \eta^{\,2} v \rangle \\
\label{eq:lem:Caccioppoli-10}
&= \langle \eta^2 [\partial_{n}v] - \Pi^{L^2} (\eta^2 [\partial_{n}v]),
\widetilde \eta^{\,2} v - \Pi^{L^2} (\widetilde \eta^{\,2} v)\rangle, 
\end{align}
where we were able to insert the cut-off function $\widetilde \eta$ 
since $\widetilde \eta \equiv 1$ on 
$\operatorname*{supp} (\eta^2 [\partial_{n}v] - \Pi^{L^2} (\eta^2 [\partial_{n}v]))\subset
\overline{(B_{(1+\varepsilon/2)R}\setminus B_{R-2h}) \cap \Gamma_{\rho}}$. 
With these observations in hand, we estimate 
\begin{align*}
\|\eta^2 [\partial_{n}v] - \Pi^{L^2} (\eta^2 [\partial_{n}v]) \|^2_{L^2(\Gamma)} &\lesssim 
\sum_{T \in \mathcal{T}_h} h^2 \|\nabla_{\Gamma} (\eta^2 [\partial_{n}v])\|^2_{L^2(T)} \\
&\lesssim C \frac{h^2}{(\varepsilon R)^2} 
\|[\partial_{n}v]\|^2_{L^2((B_{(1+\varepsilon/2)R}) \cap \Gamma_{\rho})}. 
\end{align*}
The standard approximation property 
$\|\widetilde \eta^{\,2} v - \Pi^{L^2} (\widetilde \eta^{\,2} v)\|_{L^2(\Gamma)} \lesssim 
h^{1/2} \|\widetilde \eta^{\,2} v\|_{H^{1/2}(\Gamma)}$ and 
the bound (\ref{eq:lem:estimateunjump-ii}) of 
Lemma~\ref{lem:estimateunjump} as well as the trace inequality for $\Gamma$ give 
\begin{align}
\label{eq:lem:Caccioppoli-20}
&\left| \langle \eta^2 [\partial_{n}v],v\rangle \right| 
\lesssim 
\frac{h}{\varepsilon R} 
\|[\partial_{n}v]\|_{L^2((B_{(1+\varepsilon/2)R})\cap \Gamma_{\rho})} 
h^{1/2} \|\widetilde \eta^{\,2} v\|_{H^{1/2}(\Gamma)}\nonumber  \\
&\qquad\lesssim  \frac{h}{\varepsilon R} \|\nabla v\|_{L^2(B_{(1+\varepsilon)R})} 
\|\widetilde \eta^{\,2} v\|_{H^{1/2}(\Gamma)} 
\lesssim  \frac{h}{\varepsilon R} \|\nabla v\|_{L^2(B_{(1+\varepsilon)R})} 
\|\widetilde \eta^{\,2} v\|_{H^{1}(\Omega)} .
\end{align}
The properties of $\widetilde \eta$ imply 
$\|\widetilde \eta^{\,2} v\|_{H^{1}(\Omega)} 
\lesssim \|\nabla v\|_{L^2(B_{(1+\varepsilon)R})} + 
(\varepsilon R)^{-1} \|v\|_{L^2(B_{(1+\varepsilon)R})}$. 
Inserting this into (\ref{eq:lem:Caccioppoli-20}) and the result into 
(\ref{eq:temp3}) yields 
\begin{align*}
\|\nabla (\eta v)\|^2_{L^2(B_{(1+\varepsilon)R})} &\lesssim  
\frac{h}{\varepsilon R} \| \nabla v\|^2_{L^2(B_{(1+\varepsilon)R})} 
+ \frac{h}{(\varepsilon R)^2} \|\nabla v\|_{L^2(B_{(1+\varepsilon)R})} \|v\|_{L^2(B_{(1+\varepsilon)R})}\\ 
&\phantom{\leq}\, + \frac{1}{(\varepsilon R)^2} \|v\|^2_{L^2(B_{(1+\varepsilon)R})} \\
&\lesssim \frac{h}{\varepsilon R} \|\nabla v\|^2_{L^2(B_{(1+\varepsilon)R})} 
+ \frac{1}{(\varepsilon R)^2} \|v\|^2_{L^2(B_{(1+\varepsilon)R})}, 
\end{align*}
where we employed an appropriate Young inequality in the last step
and $h/(\varepsilon R) \leq 1$. 
This implies \eqref{eq:caccioppoli1}.

{\bf Step 2:}
Starting from estimate \eqref{eq:caccioppoli1} with $\varepsilon = \frac{\delta}{2}$, we
use \eqref{eq:caccioppoli1} again with $\varepsilon = \frac{\delta}{2+\delta}$ and 
$\widetilde{R} = (1+\delta/2)R$. Since 
$\left(1+\frac{\delta}{2}\right)\left(1+\frac{\delta}{2+\delta}\right) = 1+\delta$ and 
$\frac{h}{R}\leq\frac{\delta}{16}$ implies $\frac{h}{\widetilde{R}}\leq\frac{\varepsilon}{8}$, we
arrive at 
\begin{align*}
\|\nabla v\|^2_{L^2(B_{R})} 
&\lesssim \frac{h}{\delta R} \|\nabla v\|^2_{L^2(B_{(1+\delta/2)R})} 
+ \frac{1}{(\delta R)^2} \|v\|^2_{L^2(B_{(1+\delta/2)R})} \\
&\lesssim \left(\frac{h}{\delta R}\right)^2 \|\nabla v\|^2_{L^2(B_{(1+\delta)R})} 
+\left(\frac{h}{(\delta R)^3}+\frac{1}{(\delta R)^2}\right) \|v\|^2_{L^2(B_{(1+\delta)R})},
\end{align*}
and with $h/(\delta R) \leq 1$ we conclude the proof.
\end{proof}

\subsection{Low-dimensional approximation in ${\mathcal H}_{h,0}(D,\Gamma_{\rho})$}
 
Since $\mathcal{H}_{h,0}(B_R,\Gamma_{\rho})\subset H^1(B_R)$ is 
a closed subspace by Lemma~\ref{lem:closedsubspace}, the orthogonal projection 
$\Pi_{h,R} : (H^1(B_R),\triplenorm{\cdot}_{h,R}) \rightarrow (\mathcal{H}_{h,0}(B_R,\Gamma_{\rho}), \triplenorm{\cdot}_{h,R})$ 
is well-defined.

\begin{lemma}\label{lem:lowdimapp}
Let $\delta \in (0,1)$, $R>0$ such that $\frac{h}{R}\leq \frac{\delta}{16}$ and
$B_R$, $B_{(1+\delta)R}$, $B_{(1+2\delta)R}$ be concentric boxes. Let
$\Gamma_{\rho} \subset \Gamma$ be of the form \eqref{eq:Gammarho} 
and $v\in\mathcal{H}_{h,0}(B_{(1+2\delta)R},\Gamma_{\rho})$. 
Let $\mathcal{K}_H $ be an (infinite) $\gamma$-shape regular triangulation of $\mathbb{R}^d$ 
of mesh width $H$ and assume $\frac{H}{R} \leq \frac{\delta}{4}$ for the corresponding mesh width $H$.
Let $I_H: H^1(\mathbb{R}^d) \rightarrow S^{1,1}(\mathcal{K}_H)$ be the Scott-Zhang projection. 
Then, there exists a constant $C_{\rm app} > 0$ that depends only on 
$\Omega$, $d$, and $\gamma$, such that
\begin{enumerate}
\def\theenumi{\roman{enumi}}
\item 
\label{item:lem:lowdimapp-ii}
$\big(v-\Pi_{h,R}I_H v\big)|_{B_{R}} \in \mathcal{H}_{h,0}(B_{R},\Gamma_{\rho})$; \\[-2mm]
\item 
\label{item:lem:lowdimapp-i}
$\triplenorm{v-\Pi_{h,R}I_H v}_{h,R} \leq C_{\rm app}\frac{1+2\delta}{\delta} 
\left(\frac{h}{R}+\frac{H}{R}\right)\triplenorm{v}_{h,(1+2\delta)R}$; 
\item 
\label{item:lem:lowdimapp-iii}
$\dim W\leq C_{\rm app}\left(\frac{(1+2\delta)R}{H}\right)^d$, 
where $W:=\Pi_{h,R}I_H \mathcal{H}_{h,0}(B_{(1+2\delta)R},\Gamma_{\rho}) $. 
\end{enumerate}
\end{lemma}

\begin{proof}
Since $v \in \mathcal{H}_{h,0}(B_{(1+2\delta)R},\Gamma_{\rho})$ implies 
$v \in \mathcal{H}_{h,0}(B_{R},\Gamma_{\rho})$, we have that \linebreak
$\Pi_{h,R}\left(v|_{B_{R}}\right) = v|_{B_{R}}$, which proves (\ref{item:lem:lowdimapp-ii}). 

The assumption $\frac{H}{R} \leq \frac{\delta}{4}$ implies 
$\bigcup\{K \in \mathcal{K}_H \;:\; K \cap B_{R} \neq \emptyset\} \subset B_{(1+\delta)R}$.
Then, the locality and approximation properties \eqref{eq:SZapprox} of the Scott-Zhang projection $I_H$ yield
\begin{align*}
\frac{1}{H} \norm{v - I_H v}_{L^2(B_{R})} + 
\norm{\nabla(v - I_H v)}_{L^2(B_{R})} &\lesssim  \norm{\nabla v}_{L^2(B_{(1+\delta)R})}.
\end{align*}
We apply Lemma~\ref{lem:Caccioppoli} with
$\widetilde{R} = (1+\delta)R$ and $\widetilde{\delta} = \frac{\delta}{1+\delta}$. 
Note that $(1+\widetilde{\delta})\widetilde{R} = (1+2\delta)R$.
Moreover, $16h \leq \delta R = \widetilde{\delta}\widetilde{R}$ implies 
$\frac{h}{\widetilde{R}}\leq \frac{\widetilde{\delta}}{16}$. Therefore, we get
\begin{align*}
&\triplenorm{v-\Pi_{h,R}I_H v}_{h,R}^2 =\triplenorm{\Pi_{h,R}\left(v-I_H v\right)}^2_{h,R} \leq 
\triplenorm{v-I_H v}_{h,R}^2 \\
& \qquad= \left(\frac{h}{R}\right)^{2}\norm{\nabla (v-I_H v)}_{L^2(B_{R})}^2   + 
\frac{1}{R^2} \norm{v-I_H v}_{L^2(B_{R})}^2\\
&\qquad\lesssim\frac{h^2}{R^2}\norm{\nabla v}_{L^{2}(B_{(1+\delta) R})}^2 + 
\frac{H^2}{R^2}\norm{\nabla v}_{L^2(B_{(1+\delta)R})}^2\\
  &\qquad\lesssim  \left(\frac{1+\delta}{\delta}\left(\frac{h}{R}+\frac{H}{R}\right)\right)^2\triplenorm{v}^2_{h,(1+2\delta)R},
\end{align*}
which concludes the proof of (\ref{item:lem:lowdimapp-i}).
Finally, the statement (\ref{item:lem:lowdimapp-iii}) follows from the fact that 
$\dim I_H(\mathcal{H}_{h,0}(B_{(1+2\delta)R},\Gamma_{\rho})) \lesssim ((1+2\delta)R/H)^d$. 
\end{proof}

The property (\ref{item:lem:lowdimapp-ii}) of Lemma~\ref{lem:lowdimapp} can be used to iterate
the approximation result (\ref{item:lem:lowdimapp-i}) on suitable concentric boxes. 
This will allow us to construct a subspace of $ \mathcal{H}_{h,0}(B_{(1+\kappa)R},\Gamma_{\rho})$ 
for $\kappa \in (0,1)$ with the capability to approximate at an exponential rate. 

\begin{lemma}\label{cor:lowdimapp}
Let $C_{\rm app}$ be the constant of Lemma~\ref{lem:lowdimapp}.
Let $q,\kappa \in (0,1)$, $R > 0$, $k \in \mathbb{N}$ 
and $\Gamma_{\rho} \subset \Gamma$ be of the form \eqref{eq:Gammarho}. Assume 
\begin{equation}
\label{eq:cor:lowdimapp-1}
\frac{h}{R} \leq \frac{\kappa q} {64 k \max\{1,C_{\rm app}\}}.
\end{equation}
Then, there exists a finite dimensional subspace $\widehat{W}_k$ of $\mathcal{H}_{h,0}(B_{(1+\kappa)R},\Gamma_{\rho})$ 
with dimension 
\begin{equation*}
\dim \widehat{W}_k \leq C_{\rm dim} \left(\frac{1 + \kappa^{-1}}{q}\right)^dk^{d+1}
\end{equation*}
such that for every $v \in \mathcal{H}_{h,0}(B_{(1+\kappa)R},\Gamma_{\rho})$ 
\begin{align}
\label{eq:lowdimapp}
&\min_{\widehat{w} \in \widehat{W}_k} \sqrt{h} \norm{[\partial_{n}v] - [\partial_n \widehat{w}]}_{L^2(B_R\cap\Gamma_{\rho})}
 \leq
\min_{\widehat{w} \in \widehat{W}_k} \sqrt{h} \norm{[\partial_{n}v] - [\partial_n \widehat{w}]}_{L^2(B_R\cap\Gamma)} 
 \\
&\qquad \qquad\leq  C_{\rm low}\frac{R}{h} \min_{\widehat{w}\in \widehat{W}_k} \triplenorm{v-\widehat{w}}_{h,(1+\kappa/2)R} \leq 
C_{\rm low}\frac{R}{h} q^{k} \triplenorm{v}_{h,(1+\kappa)R}.
\nonumber 
\end{align}
The constants $C_{\rm dim}$, $C_{\rm low}>0$ depend only on $\Omega$, $d$, and the 
$\gamma$-shape regularity of the quasiuniform triangulation $\mathcal{T}_h$.
\end{lemma}

\begin{proof}
Let $B_R$ and $B_{(1+\delta_j)R}$, 
with $\delta_j := \kappa(1-\frac{j}{2k})$ for $j=0,\dots,k$ be concentric boxes. 
We have $\kappa = \delta_0>\delta_1 > \dots >\delta_{k} = \frac{\kappa}{2}$.
In the following, we iterate the approximation result of Lemma~\ref{lem:lowdimapp} 
on the boxes $B_{(1+\delta_j)R}$. 
Choosing $H = \frac{\kappa q R}{64k\max\{C_{\rm app},1\}}$, we have $h \leq H$. 
We apply Lemma~\ref{lem:lowdimapp} with 
$\widetilde{R}_j = (1+\delta_j)R$ and 
$\widetilde{\delta}_j = \frac{\kappa}{4k(1+\delta_j)}<\frac{1}{4}$.
Note that $\delta_{j-1} = \delta_j +\frac{\kappa}{2k}$ gives 
$(1+\delta_{j-1})R=(1+2\widetilde{\delta}_j)\widetilde{R}_j$ and 
our choice of $H$ implies 
$\frac{H}{\widetilde{R}_j}\leq \frac{\widetilde{\delta}_j}{4}$. 
Hence, for $j=1$, Lemma~\ref{lem:lowdimapp} provides an approximation $w_1$ in a subspace
 $W_1$ of $\mathcal{H}_{h,0}(B_{(1+\delta_1)R},\Gamma_{\rho})$ 
with $\dim W_1 \leq C_{\rm app}\left(\frac{(1+2\widetilde{\delta}_1)\widetilde{R}_1}{H}\right)^d 
= C_{\rm app}\left(\frac{(1+\kappa)R}{H}\right)^d$, 
satisfying
\begin{align*}
\triplenorm{v-w_1}_{h,(1+\delta_1)R}  &= \triplenorm{v-w_1}_{h,\widetilde{R}_1} \leq
C_{\rm app}{\frac{1+2\widetilde{\delta}_1}{\widetilde{\delta}_1}}
\left(\frac{h}{\widetilde{R}_1}+\frac{H}{\widetilde{R}_1}\right)
 \triplenorm{v}_{h,(1+2\widetilde{\delta}_1)\widetilde{R}_1} \\
&\leq 2C_{\rm app}\frac{H}{(1+\delta_1)R}{\frac{1+2\widetilde{\delta}_1}{\widetilde{\delta}_1}} 
\triplenorm{v}_{h,(1+\delta_0)R}\\
 &= 8C_{\rm app}\frac{k H}{\kappa R}(1+2\widetilde{\delta}_1)\triplenorm{v}_{h,(1+\kappa)R} 
\leq  q\triplenorm{v}_{h,(1+\kappa)R}.
\end{align*}
Since $v|_{B_{(1+\delta_1)R}}-w_1 \in \mathcal{H}_{h,0}(B_{(1+\delta_1)R},\Gamma_{\rho})$, 
Lemma~\ref{lem:lowdimapp} can be applied to $v-w_1$,
and provides an approximation $w_2$ of $v-w_1$ in a subspace $W_2$ 
of $\mathcal{H}_{h,0}(B_{(1+\delta_2)R},\Gamma_{\rho})$ with
$\dim W_2\leq C_{\rm app}\left(\frac{(1+\kappa)R}{H}\right)^d$. Arguing as for $j=1$, we get
\begin{equation*}
\triplenorm{v-w_1-w_2}_{h,(1+\delta_2)R} \leq q \triplenorm{v-w_1}_{h,(1+\delta_1)R} \leq  
q^2 \triplenorm{v}_{h,(1+\kappa)R}.
\end{equation*}
Continuing this process $k-2$ times, one obtains an approximation $\widehat{w} := \sum_{j=1}^kw_i$ 
in the space $\widehat{W}_k := \sum_{j=1}^{k}W_j$ 
of dimension $\dim \widehat{W}_k \leq C_{\rm app}k\left(\frac{(1+\kappa)R}{H}\right)^d \leq
C_{\rm dim} ((1+\kappa^{-1}) q^{-1})^dk^{d+1}$ with
\begin{equation*}
\triplenorm{v-\widehat{w}}_{h,(1+\kappa/2)R} = \triplenorm{v-\widehat{w}}_{h,(1+\delta_k)R}
\leq q^k \triplenorm{v}_{h,(1+\kappa)R}.
\end{equation*} 
Finally, since \eqref{eq:cor:lowdimapp-1} ensures $h/R \leq \kappa/8$, we may use 
Lemma~\ref{lem:estimateunjump} to estimate 
\begin{equation*}\sqrt{h} \|[\partial_{n}v] - [\partial_n \widehat w]\|_{L^2(B_R \cap \Gamma)} 
\leq C\norm{\nabla(v-\widehat{w})}_{L^2(B_{(1+\kappa/2)R})}
\leq C \frac{R}{h} \triplenorm{v - \widehat w}_{h,(1+\kappa/2)R}\end{equation*} to conclude the argument. 
\end{proof}

Now we are able to prove the main result of this section, Proposition~\ref{thm:function-approximation}.

\begin{proof}[Proof of Proposition~\ref{thm:function-approximation}]
Due to the admissibility condition and the
definition of bounding boxes, we have 
$\operatorname{dist}(B_{R_\tau},B_{R_\sigma}) 
\ge \eta^{-1} \operatorname*{diam} (B_{R_\tau}) =  \eta^{-1}\sqrt{d} R_{\tau}$. 
The choice $\kappa = \frac{1}{1+\eta}$ implies 
\begin{align*}
\operatorname{dist}(B_{(1+\kappa) R_\tau},B_{R_\sigma}) &\geq 
\operatorname{dist}(B_{R_\tau},B_{R_\sigma}) - \kappa R_{\tau} \sqrt{d} \\
 &\geq \sqrt{d}R_{\tau}(\eta^{-1}-\kappa) = 
\sqrt{d}R_{\tau}\left(\frac{1}{\eta} - \frac{1}{1+\eta}\right) >0.
\end{align*}
The potential $u = \widetilde{V}\phi_h$ with the Galerkin solution $\phi_h \in S^{0,0}(\mathcal{T}_h)$ 
of (\ref{eq:model}) then satisfies 
$u \in \mathcal{H}_{h,0}(B_{(1+\kappa) R_\tau},\Gamma)$. 
The inverse estimate $\sqrt{h}\norm{\phi_h}_{L^2(\Gamma)} \lesssim \norm{\phi_h}_{H^{-1/2}(\Gamma)}$ 
of 
\eqref{eq:inverse-a}
and the ellipticity of the simple-layer operator
 as well as the discrete boundary integral equation \eqref{eq:model} provide
\begin{align*}
\norm{\phi_h}_{H^{-1/2}(\Gamma)}^2 &\lesssim  \skp{V\phi_h,\phi_h} =  \skp{f,\phi_h}
= \skp{\Pi^{L^2} f, \phi_h}
 \lesssim \norm{\Pi^{L^2}f}_{L^2(\Gamma)}\norm{\phi_h}_{L^2(\Gamma)} \\
&\lesssim h^{-1/2}\norm{\Pi^{L^2}f}_{L^2(\Gamma)}\norm{\phi_h}_{H^{-1/2}(\Gamma)}.
\end{align*}
Then, the boundedness of $\widetilde{V}: H^{-1/2}(\Gamma) \rightarrow H^1_{\text{loc}}(\mathbb{R}^d)$ and $\frac{h}{R_{\tau}}<1$ lead to
\begin{align*}
\triplenorm{\widetilde{V} \phi_h}_{h,R_{\tau}(1+\kappa)} 
&\leq 2\left(1+\frac{1}{R_\tau}\right)\norm{\widetilde{V} \phi_h}_{H^1(B_{2R_{\tau}})} 
\lesssim \left(1+\frac{1}{R_\tau}\right) \norm{\phi_h}_{H^{-1/2}(\Gamma)} \\
&\lesssim  \left(1+\frac{1}{R_\tau}\right) h^{-1/2}\norm{\Pi^{L^2}f}_{L^2(\Gamma)}.
\end{align*}
After these preparations, we are in a position to define the space $W_k$, for which we distinguish two cases. 

{\bf Case 1:} The condition \eqref{eq:cor:lowdimapp-1} is satisfied with $R = R_{\tau}$. 

With the space $\widehat{W}_k$ provided by Lemma~\ref{cor:lowdimapp} we set 
$W_k := \{[\partial_n \widehat{w}] : \widehat{w} \in \widehat{W}_k\}$. 
Then, Lemma~\ref{cor:lowdimapp} and $R_{\tau}\leq 2\operatorname*{diam}(\Omega)$ lead to
 \begin{align*}
\min_{w\in W_k}\norm{\phi_h-w}_{L^2(B_{R_{\tau}}\cap \Gamma)} &\lesssim  
\frac{R_{\tau}}{h^{3/2}} q^k \triplenorm{\widetilde{V}\phi_h}_{h,(1+\kappa)R_{\tau}} \lesssim 
(R_{\tau}+1)h^{-2}q^k\norm{\Pi^{L^2}f}_{L^2(\Gamma)} \\
 &\lesssim h^{-2} q^k \norm{\Pi^{L^2}f}_{L^2(\Gamma)},
\end{align*}
and the dimension of $W_k$ is bounded by 
\begin{equation*}
\dim W_k \leq C_{\rm dim} \left(\frac{1+\kappa^{-1}}{q}\right)^d k^{d+1} = C_{\rm dim} (2 + \eta)^d q^{-d} k^{d+1}. 
\end{equation*}
{\bf Case 2:} The condition \eqref{eq:cor:lowdimapp-1} is not satisfied. 

Then, we select 
$W_k:= \left\{w|_{B_{R_\tau}\cap \Gamma} : w \in S^{0,0}({\mathcal T}_h)\right\}$ and 
the minimum in \eqref{eq:thm:function-approximation-1} is obviously zero. 
By the choice of $\kappa$ and $\frac{h}{R}>\frac{\kappa q}{64k\max\{1,C_{\rm app}\}}$, 
the dimension of $W_k$ is bounded by 
\begin{align*}
\dim W_k \lesssim \left(\frac{R_{\tau}}{h}\right)^{d-1} 
&\lesssim \left(\frac{64 k \max\{C_{\rm app},1\}}{\kappa q}\right)^{d-1} 
\simeq \left((1+\eta) q^{-1} k \right)^{d-1} \\
&\lesssim (2+\eta)^d q^{-d} k^{d+1}. 
\end{align*}
This concludes the proof of the first inequality in \eqref{eq:thm:function-approximation-1}.  
The second inequality in \eqref{eq:thm:function-approximation-1} follows 
from the $L^2(\Gamma)$-stability of the $L^2(\Gamma)$-orthogonal projection.
\end{proof}

\section{Proof of the approximation results for ${\mathbf V}^{-1}$}
\label{sec:H-matrix-approximation}
In this section, the approximation result given in Proposition~\ref{thm:function-approximation} 
is used to construct a low-rank approximation of a matrix block 
$\mathbf{V}^{-1}|_{\tau\times\sigma}$ and in turn an $\mathcal{H}$-matrix approximation 
of $\mathbf{V}^{-1}$. This is achieved with local variants of the isomorphism \eqref{eq:isomorphism}, 
and our arguments follow the lines of \cite[Theorem~{2}]{Boerm}. 

\begin{proof}[Proof of Theorem~\ref{th:blockapprox}]
If $C_{\rm dim} (2+\eta)^d q^{-d} k^{d+1} \geq \min (\abs{\tau},\abs{\sigma})$, we use the exact matrix block 
$\mathbf{X}_{\tau\sigma}=\mathbf{V}^{-1}|_{\tau \times \sigma}$ and 
$\mathbf{Y}_{\tau\sigma} = \mathbf{I} \in \mathbb{R}^{\abs{\sigma}\times\abs{\sigma}}$. 

If $C_{\rm dim}(2+\eta)^d q^{-d} k^{d+1}  < \min (\abs{\tau},\abs{\sigma})$, we employ the approximation
result of Proposition~\ref{thm:function-approximation} in the following way.
For the cluster $\tau \subset \mathcal{I}$, we define \linebreak
$\mathbb{R}^{\tau} := \left\{\mathbf{x} \in\mathbb{R}^N : x_i = 0\, \forall i \notin \tau\right\}$ 
and the mappings
 \begin{equation*}
\Phi_{\tau}: \mathbb{R}^{\tau} \rightarrow S^{0,0}(\mathcal{T}_h),\quad \mathbf{x} \mapsto \sum_{j\in\tau} x_j \chi_j, \quad \text{and} \quad
 \Lambda_{\tau} : L^2(\Gamma) \rightarrow \mathbb{R}^{\abs{\tau}},\quad w \mapsto (\overline{w}_i)_{i \in \tau},
\end{equation*}
 where $\overline{w}_i$ denotes the mean value on the element $T_i \in \mathcal{T}_h$. 
Hence, for a piecewise constant function the mapping $\Lambda_{\tau}$ returns 
the constant value on each element corresponding to the cluster $\tau$. 
Moreover, $\Phi_{\tau}\Lambda_{\tau}$ is the restriction of the $L^2$-projection
onto $S^{0,0}(\mathcal{T}_h)$ to 
$\Gamma_{\tau}:=\text{interior}\left(\bigcup_{i \in \tau} \overline{T_i}\right) \subset B_{R_{\tau}}$. Thus, in particular, 
for a piecewise constant function $\widetilde{\phi}\in S^{0,0}(\mathcal{T}_h)$ we get 
$\Phi_{\tau}(\Lambda_{\tau} \widetilde{\phi}) = \widetilde{\phi}|_{\Gamma_{\tau}}$. 
For $\mathbf{x} \in \mathbb{R}^{\tau}$, \eqref{eq:isomorphism} implies
\begin{equation*}
Ch^{d/2}\norm{\mathbf{x}}_{2} \leq \norm{\Phi_\tau(\mathbf{x})}_{L^2(\Gamma)} \leq 
\widetilde{C} h^{d/2}\norm{\mathbf{x}}_{2}, \quad \mathbf{x} \in \mathbb{R}^{{\tau}}. 
\end{equation*}
The adjoint $\Lambda_{\mathcal{I}}^* : \mathbb{R}^N \rightarrow L^2(\Gamma)'\simeq L^2(\Gamma), 
\mathbf{b} \mapsto \sum_{i\in \mathcal{I}} b_i (u \mapsto \overline{u}_i)$ 
of $\Lambda_{\mathcal{I}}$ satisfies, because of \eqref{eq:isomorphism} and the $L^2$-stability of $\Phi_{\mathcal{I}} \Lambda_{\mathcal{I}}$,
\begin{align*}
\norm{\Lambda_{\mathcal{I}}^* \mathbf{b}}_{L^2(\Gamma)} &= \sup_{w \in L^2(\Gamma)}
\frac{\skp{\mathbf{b},\Lambda_{\mathcal{I}} w}_2}{\norm{w}_{L^2(\Gamma)}} 
  \lesssim \norm{\mathbf{b}}_2 \sup_{w \in L^2(\Gamma)}
\frac{h^{-d/2}\norm{\Phi_{\mathcal{I}}\Lambda_{\mathcal{I}} w}_{L^2(\Gamma)}}{\norm{w}_{L^2(\Gamma)}}\\
 &\lesssim h^{-d/2}\norm{\mathbf{b}}_2. 
\end{align*}
Let $\mathbf{b} \in \mathbb{R}^N$. Defining $f := \Lambda_{\mathcal{I}}^*\mathbf{b}|_{\sigma}$,
we get $b_i =\skp{f,\chi_i}$ for $i \in \sigma$, and $\operatorname*{supp} f \subset B_{R_{\sigma}}\cap\Gamma$.
Proposition~\ref{thm:function-approximation} provides a 
finite dimensional space $W_k$ and an element $w \in W_k$ 
that is a good approximation to the Galerkin solution $\phi_h|_{B_{R_{\tau}}\cap \Gamma}$.
It is important to note that the space $W_k$ is constructed independently of the function $f$; 
it depends only on the cluster pair $(\tau,\sigma)$. 
The estimate \eqref{eq:isomorphism}, the approximation result from 
Proposition~\ref{thm:function-approximation},
and $\norm{\Pi^{L^2}f}_{L^2(\Gamma)} = \norm{\Lambda_{\mathcal{I}}^*\mathbf{b}}_{L^2(\Gamma)} 
\lesssim h^{-d/2}\norm{\mathbf{b}}_2$ imply
\begin{align*}
\norm{\Lambda_{\tau} \phi_h - \Lambda_{\tau} w}_2 &\lesssim 
h^{-d/2}\norm{\Phi_{\tau}(\Lambda_{\tau} \phi_h-\Lambda_{\tau} w)}_{L^2(\Gamma)} 
\lesssim h^{-d/2}\norm{\phi_h-w}_{L^2(B_{R_{\tau}}\cap \Gamma)} \\
  &\lesssim h^{-d/2-2}q^k\norm{\Pi^{L^2}f}_{L^2(\Gamma)} \lesssim h^{-(d+2)}q^k\norm{\mathbf{b}}_{2}.
\end{align*}
In order to translate this approximation result to the matrix level, let 
\begin{equation*}
\mathcal{W} := \{\Lambda_{\tau} w \; : \; w \in W_k \}.
\end{equation*}
Let the columns of $\mathbf{X}_{\tau\sigma}$ be an orthogonal basis of the space $\mathcal{W}$. 
Then, the rank of $\mathbf{X}_{\tau\sigma}$ is bounded by $C_{\rm dim}  (2+\eta)^d q^{-d}k^{d+1} $. 
Since $\mathbf{X}_{\tau\sigma} \mathbf{X}_{\tau\sigma}^T$ is the orthogonal projection from 
$\mathbb{R}^N$ onto $\mathcal{W}$, 
we get that $z:=\mathbf{X}_{\tau\sigma} \mathbf{X}_{\tau\sigma}^T \Lambda_{\tau} \phi_h$ 
is the best approximation of $\Lambda_{\tau} \phi_h$ in $\mathcal{W}$ and arrive at 
\begin{equation} 
\label{eq:matrix-level-estimate-function}
\norm{\Lambda_{\tau} \phi_h-z}_2 \leq \norm{\Lambda_{\tau} \phi_h-\Lambda_{\tau} w}_2 \lesssim 
h^{-(d+2)}q^k\norm{\mathbf{b}}_2
\simeq N^{(d+2)/(d-1)}q^k\norm{\mathbf{b}}_2.
\end{equation}
Noting that $\Lambda_{\tau} \phi_h = \mathbf{V}^{-1}|_{\tau\times \sigma} \mathbf{b}|_\sigma$,
if we define $\mathbf{Y}_{\tau,\sigma} := \mathbf{V}^{-1}|_{\tau\times \sigma}^{T}\mathbf{X}_{\tau\sigma}$, 
we thus get $z = \mathbf{X}_{\tau\sigma} \mathbf{Y}_{\tau\sigma}^T \mathbf{b}|_\sigma$. The bound 
\eqref{eq:matrix-level-estimate-function} expresses 
\begin{equation}
\label{eq:matrix-level-estimate-function-1}
\| 
\left( \mathbf{V}^{-1}|_{\tau\times \sigma} - \mathbf{X}_{\tau\sigma} \mathbf{Y}_{\tau\sigma}^T \right)\mathbf{b}|_\sigma
\|_2 \lesssim N^{(d+2)/(d-1)}q^k\norm{\mathbf{b}}_2.
\end{equation}
The space $W_k$ depends only on the cluster pair $(\tau,\sigma)$ and the estimate 
\eqref{eq:matrix-level-estimate-function-1} is valid for any $\mathbf b$. 
This concludes the proof. 
\end{proof}

The following lemma gives an estimate for the global spectral norm by the local spectral norms.

\begin{lemma}[{\cite{GrasedyckDissertation},\cite[Lemma 6.5.8]{HackbuschBuch}}]\label{lem:spectralnorm}
Let $\mathbf{M} \in \mathbb{R}^{N\times N}$ and $P$ be a partitioning of ${\mathcal{I}}\times {\mathcal{I}}$. Then,
\begin{equation*}
\norm{\mathbf{M}}_2 \leq C_{\rm sp} \left(\sum_{\ell=0}^{\infty}\max\{\norm{\mathbf{M}|_{\tau\times \sigma}}_2 : 
(\tau,\sigma) \in P, \operatorname{level}(\tau) = \ell\}\right),
\end{equation*}
where the sparsity constant $C_{\rm sp}$ is defined in \eqref{eq:sparsityConstant}.
\end{lemma}

Now we are able to prove our main result, Theorem~\ref{th:Happrox}.

\begin{proof}[Proof of Theorem \ref{th:Happrox}]
Theorem \ref{th:blockapprox} provides matrices $\mathbf{X}_{\tau\sigma} \in \mathbb{R}^{\abs{\tau}\times r}$, 
$\mathbf{Y}_{\tau\sigma} \in \mathbb{R}^{\abs{\sigma}\times r}$, 
so we can define the $\mathcal{H}$-matrix $\mathbf{V}_{\mathcal{H}}$ by 
\begin{equation*}
\mathbf{W}_{\mathcal{H}} = \left\{
\begin{array}{l}
 \mathbf{X}_{\tau\sigma}\mathbf{Y}_{\tau\sigma}^T \quad \;\textrm{if}\hspace{2mm} (\tau,\sigma) \in P_{\text{far}}, \\
 \mathbf{V}^{-1}|_{\tau \times \sigma} \quad \textrm{otherwise}.
 \end{array}
 \right.
\end{equation*}
On each admissible block $(\tau,\sigma) \in P_{\rm far}$ we can use the blockwise estimate of Theorem \ref{th:blockapprox} and get
\begin{equation*}
\norm{(\mathbf{V}^{-1} - \mathbf{W}_{\mathcal{H}})|_{\tau \times \sigma}}_2 \leq C_{\rm apx}N^{(d+2)/(d-1)} q^k.
\end{equation*}
On inadmissible blocks, the error is zero by definition. 
Therefore, Lemma \ref{lem:spectralnorm} leads to
\begin{align*}
\norm{\mathbf{V}^{-1} - \mathbf{W}_{\mathcal{H}}}_2 &\leq C_{\rm sp} \left(\sum_{\ell=0}^{\infty}\text{max}
\{\norm{(\mathbf{V}^{-1} - \mathbf{V}_{\mathcal{H}})|_{\tau \times \sigma}}_2 : (\tau,\sigma) \in P, 
\operatorname{level}(\tau) = \ell\}\right) \\
 &\leq C_{\rm apx}C_{\rm sp} N^{(d+2)/(d-1)} q^k \operatorname{depth}(\mathbb{T}_{\mathcal{I}}).
\end{align*}
With $r = C_{\rm dim}(2+\eta)^dq^{-d}k^{d+1}$,
defining $b = -\frac{\ln(q)}{C_{\rm dim}^{1/(d+1)}}q^{d/(d+1)}(2+\eta)^{-d/(1+d)} > 0$ leads to 
$q^k = e^{-br^{1/(d+1)}}$, and hence
\begin{equation*}
\norm{\mathbf{V}^{-1} - \mathbf{W}_{\mathcal{H}}}_2 \leq C_{\rm apx}C_{\rm sp} N^{(d+2)/(d-1)} 
{\rm depth}(\mathbb{T}_{\mathcal{I}})e^{-br^{1/(d+1)}},
\end{equation*}
which concludes the proof.
\end{proof}

\section{$\mathcal{H}$-Cholesky decomposition: Proof of Theorem~\ref{th:HLU}}
\label{sec:LU-decomposition}
The aim of this section is the proof of Theorem~\ref{th:HLU}. 
Our procedure follows \cite{Bebendorf07,GrasedyckKriemannLeBorne,FMP13} and is based on showing 
that the off-diagonal block of certain Schur complements can be approximated by low-rank matrices.  
The analysis of these Schur complement matrices in Section~\ref{sec:schur-complements} is therefore
the main contribution of the section. 

The matrix $\mathbf{V}$ is symmetric and positive definite and therefore has a (classical) Cholesky-decomposition 
$\mathbf{V}=\mathbf{C}\mathbf{C}^T$, 
where $\mathbf{C}$ is a lower triangular matrix.
Moreover, the existence of the Cholesky decomposition does not depend on the numbering of the degrees of freedom, 
i.e., for every other numbering of the basis functions there is a Cholesky decomposition as well 
(see, e.g., \cite[Cor.~{3.5.6}]{horn-johnson13}). 
The existence of the Cholesky decomposition implies
the invertibility of the matrix $\mathbf{V}|_{\rho \times \rho}$ for any $n \leq N$ and index set 
$\rho := \{1,\ldots,n\}$ (see, e.g., \cite[Cor.~{3.5.6}]{horn-johnson13}). 

The first step is the approximation of appropriate Schur complements. 

\subsection{Schur complements}
\label{sec:schur-complements}

For a cluster pair $(\tau,\sigma)$, we define the index set $\rho := \{i\in \mathcal{I} : i < \min(\tau\cup\sigma)\}$
and the Schur complement
\begin{equation}\label{eq:defSchur}
\mathbf{S}(\tau,\sigma) = \mathbf{V}|_{\tau\times\sigma} - \mathbf{V}|_{\tau\times \rho} (\mathbf{V}|_{\rho\times \rho})^{-1}\mathbf{V}|_{\rho\times\sigma}.
\end{equation}
One way to approximate the Schur complement is to use the $\mathcal{H}$-arithmetic.
As stated in \cite[Theorem 15]{GrasedyckKriemannLeBorne}, 
this results in a low-rank approximation to 
$\mathbf{S}(\tau,\sigma)$ of rank  
$C_{\rm id}C_{\rm sp}({\rm depth}(\mathbb{T}_{\mathcal{I}})+1)^2 r$, where the idempotency constant 
$C_{\rm id}$ is defined in \cite{GrasedyckHackbusch}, and
$r$ is the block rank used for the approximation of the inverse matrix $\mathbf{V}^{-1}$. 
In the following Theorem~\ref{lem:Schur}, we provide a low-rank approximation by using a different approach,
which uses the techniques developed in Section~\ref{sec:Approximation-solution} and gives a better bound in terms of the
rank of the approximation, i.e., a rank of $Cr$ is sufficient to obtain the same accuracy.
This approach relies on interpreting Schur complements as BEM matrices from certain constrained spaces. 

The key step is Theorem~\ref{lem:Schur} below. For its proof, we need a degenerate approximation 
of the Green's function $G(\cdot,\cdot)$.
This is a classical result that underlies the log-linear matrix-vector multiplication in BEM 
and can be achieved by multipole expansions \cite{rokhlin85,greengard-rokhlin97}, 
Taylor expansions \cite{hackbusch-nowak88,hackbusch-nowak89,sauter92,hackbusch-sauter93} or by interpolation 
(see, e.g., \cite[Sec.~{7.1.3.1}]{SauterSchwab}). The following lemma recalls a variant of such a degenerate 
approximation that is obtained with Chebyshev interpolation: 

\begin{lemma}\label{lem:lowrankGreensfunction}
Let $\widetilde \eta>0$ and fix $\eta^\prime \in (0,2 \widetilde \eta)$. Then, for every hyper cube 
$B_Y \subset \mathbb{R}^d$, $d \in \{2,3\}$ and closed $D_X \subset \mathbb{R}^d$ with 
$\operatorname{dist}(B_Y,D_X)\geq \widetilde\eta \operatorname*{diam}(B_Y)$ the following is true: For every 
$r \in \mathbb{N}$ there exist functions $g_{1,i}$, $g_{2,i}$, $i=1,\ldots,r$, such that 
\begin{equation*}
\norm{G(x,\cdot)-\sum_{i=1}^r g_{1,i}(x) g_{2,i}(\cdot)}_{L^{\infty}(B_{Y})}
\leq C \frac{(1+1/\widetilde \eta)}{\operatorname{dist}(\{x\},B_Y)^{d-2}} (1 + \eta^\prime)^{-r^{1/d}}
\quad \forall x \in D_X, 
\end{equation*}
for a constant $C$ that depends solely on the choice of $\eta^\prime  \in (0,2\widetilde\eta)$. 
\end{lemma}
\begin{proof} 
Let $I^y_k:C(\overline{B_Y}) \rightarrow {\mathcal Q}_k$ 
be the tensor product interpolation operator of degree $k$ 
defined on $C(\overline{B_Y})$ and mapping into the 
space ${\mathcal Q}_k$ of polynomial of degree $k$
in each variable. Note that $\dim {\mathcal Q}_k = (k+1)^d=:r$. 
The approximation $G_r(x,y):= \sum_{i=1}^r g_{1,i}(x) g_{2,i}(y)$ is then taken to be 
$
G_r(x,\cdot):= I^y_k G(x,\cdot). 
$
The stated error bound follows from estimates for Chebyshev interpolation. 
We note that the Green's function for the Laplacian is asymptotically smooth 
(see \cite[Definition~{4.2.5}]{HackbuschBuch} with constant $c_{\rm as}(\nu) = C\nu!$).
Tensorial interpolation in the form given in \cite{BoermGrasedyck04}
allows us to estimate 
\begin{align*}
\norm{G(x,\cdot)-I^y_k G(x,\cdot)}_{L^{\infty}(B_Y)} 
&\lesssim \frac{1}{\operatorname{dist}(\{x\},B_Y)^{d-2}}
\left(1+\frac{\operatorname*{diam}(B_Y)}{\operatorname{dist}(B_Y,\{x\})}\right)\Lambda_k^d r^{1/d}\\
&\quad \cdot\left(1+\frac{2\operatorname{dist}(B_Y,\{x\})}{\operatorname*{diam}(B_Y)}\right)^{-r^{1/d}}, 
\end{align*}
where $\Lambda_k \leq 1+ \frac{2}{\pi}\ln (k+1)$ is the Lebesgue constant of Chebyshev interpolation, cf. \cite{Rivlin}. 
The observation $\operatorname{dist}(\{x\},B_Y) \ge \operatorname{dist}(B_Y,D_X)\geq \widetilde\eta \operatorname*{diam}(B_Y)$ 
and the choice $\eta^\prime < 2 \widetilde \eta$  imply the claimed estimate. 
\end{proof}

\begin{theorem}\label{lem:Schur}
Let $(\tau,\sigma)$ be an $\eta$-admissible cluster pair, 
set $\rho := \{i\in \mathcal{I} : i < \min(\tau\cup\sigma)\}$,
and let the Schur complement $\mathbf{S}(\tau,\sigma)$ be defined in \eqref{eq:defSchur}.
Then, there exists a rank-$r$ matrix $\mathbf{S}_{r}(\tau,\sigma)$ such that 
\begin{equation*}
\norm{\mathbf{S}(\tau,\sigma) - \mathbf{S}_{r}(\tau,\sigma)}_2 \leq C_{\rm sc}  N^{3/(2d-2)} e^{-br^{1/(d+1)}}\norm{\mathbf{V}}_2,
\end{equation*}
where the constants $C_{\rm sc},b >0$ depend only on $\Omega$,
$d$, the $\gamma$-shape regularity of the quasiuniform triangulation $\mathcal{T}_h$, and $\eta$.
\end{theorem}

\begin{proof}
Let $B_{R_{\tau}},B_{R_{\sigma}}$ be
bounding boxes for the clusters $\tau$, $\sigma$ satisfying \eqref{eq:admissibility}.
We define $\Gamma_{\rho} = {\rm interior}\left( \bigcup_{i\in \rho}\operatorname*{supp} \psi_i \right) \subset \Gamma$. 
First, we recall that the Schur complement matrix
$\mathbf{S}(\tau,\sigma)$ can be understood in terms of an orthogonalization
with respect to the degrees of freedom in $\rho$. More precisely, 
a direct calculation (see for the essentials, e.g., \cite{Brenner99}) 
shows for $\boldsymbol{\phi}\in\mathbb{R}^{\abs{\tau}}$, $\boldsymbol{\psi}\in\mathbb{R}^{\abs{\sigma}}$
the  representation
\begin{equation}\label{eq:SchurRepresentation}
\boldsymbol{\phi}^T \mathbf{S}(\tau,\sigma)\boldsymbol{\psi} = \skp{V\widetilde{\phi},\psi}, 
\end{equation}
with the following relation between the functions $\psi$, $\widetilde\phi$ and the vectors 
$\boldsymbol{\psi}$, $\boldsymbol{\phi}$, respectively: 
$\psi = \sum_{j=1}^{\abs{\sigma}}\boldsymbol{\psi}_j\chi_{j_{\sigma}}$, 
where the index $j_{\sigma}$ denotes the $j$-th basis function corresponding to the cluster $\sigma$,
and the function $\widetilde{\phi} \in S^{0,0}(\mathcal{T}_h)$ 
is defined by $\widetilde{\phi} = \phi + \phi_{\rho}$ 
with $\phi = \sum_{j=1}^{\abs{\tau}}\boldsymbol{\phi}_j\chi_{j_{\tau}}$
and $\operatorname*{supp} \phi_{\rho} \subset \overline{\Gamma_{\rho}}$ such that 
\begin{equation}\label{eq:SchurOrthogonality}
\skp{V\widetilde{\phi},\widehat{\psi}} = 0 \quad \forall \widehat{\psi} \in S^{0,0}({\mathcal T}_h) \; \text{with}\; 
\operatorname*{supp} \widehat{\psi} \subset \overline{\Gamma_{\rho}}.
\end{equation}
Our low-rank approximation of the Schur complement matrix $\mathbf{S}(\tau,\sigma)$ will have two
ingredients: first, 
based on the the techniques of Section~\ref{sec:Approximation-solution} we exploit the 
orthogonality \eqref{eq:SchurOrthogonality} to 
construct a low-dimensional space $\widehat W_k$ from which for any $\phi$, the corresponding
function $\widetilde \phi$ can be approximated well. Second,  
we exploit that the function $\psi$ in (\ref{eq:SchurRepresentation}) is 
supported by $\Gamma_{\sigma}$, and we will use Lemma~\ref{lem:lowrankGreensfunction}. 

Let $\delta = \frac{1}{1+\eta}$ and $B_{R_{\sigma}}$, $B_{(1+\delta)R_{\sigma}}$ be concentric boxes.
The symmetry of $V$ leads to
\begin{align}\label{eq:tmpSchur}
\skp{V\widetilde{\phi},\psi} &= \skp{\widetilde{\phi},V\psi} \nonumber \\ &=
\skp{\widetilde{\phi},V\psi}_{L^2( B_{(1+\delta)R_{\sigma}}\cap\Gamma_{\rho} )} + 
\skp{\widetilde{\phi},V\psi}_{L^2(\Gamma \backslash B_{(1+\delta)R_{\sigma}})}. 
\end{align}
First, we treat the first term on the right-hand side of \eqref{eq:tmpSchur}.
The choice of $\delta$ and the admissibility condition \eqref{eq:admissibility}, where 
we can assume $\min\{\operatorname*{diam}(B_{R_{\tau}}),\operatorname*{diam}(B_{R_{\sigma}})\} = \sqrt{d}R_{\sigma}$ 
due to the symmetry $\mathbf{S}(\tau,\sigma) = \mathbf{S}(\sigma,\tau)^T$,
imply
\begin{equation*}
\operatorname{dist}(B_{(1+2\delta)R_{\sigma}},B_{R_{\tau}}) \geq \operatorname{dist}(B_{R_{\sigma}},B_{R_{\tau}})-\sqrt{d}\delta R_{\sigma}
\geq \sqrt{d}R_{\sigma}(\eta^{-1}-\delta) > 0.
\end{equation*}
Therefore, we have $\widetilde{\phi}|_{B_{(1+2\delta)R_{\sigma}}\cap\Gamma_{\rho}} = 
\phi_{\rho}|_{B_{(1+2\delta)R_{\sigma}}\cap\Gamma_{\rho}}$
and the orthogonality \eqref{eq:SchurOrthogonality} holds 
on the box $B_{(1+2\delta)R_{\sigma}}$. Thus, by definition of $\mathcal{H}_{h,0}$, we have 
that the potential  
$\widetilde{V}\widetilde{\phi} \in \mathcal{H}_{h,0}(B_{(1+2\delta)R_{\sigma}},\Gamma_{\rho})$.

As a consequence, Lemma~\ref{cor:lowdimapp} can be applied to $\widetilde{V}\widetilde{\phi}$
with $R := (1+\delta)R_{\sigma}$ and $\kappa := \frac{1}{2+\eta} = \frac{\delta}{1+\delta}$. Note that
$(1+\kappa)(1+\delta) = 1+2\delta$ and $1+\kappa^{-1} = 3+\eta$.
Hence, we get a low dimensional space $\widehat{W}_k$ of dimension
$\dim \widehat{W}_k \leq C_{\rm dim}(3+\eta)^dq^{-d}k^{d+1} =: r$, and
the best approximation $\widehat{\phi} = \Pi_{\widehat{W}_k}\widetilde{\phi}$
to $\widetilde{\phi}$ from the space $\widehat{W}_k$ satisfies
\begin{align*}
\norm{\widetilde{\phi}-\widehat{\phi}}_{L^2(B_{(1+\delta)R_{\sigma}}\cap\Gamma_{\rho})} &\lesssim 
R_{\sigma} h^{-3/2}q^k \triplenorm{\widetilde{V}\widetilde{\phi}}_{h,(1+2\delta)R_{\sigma}} \\
&\lesssim  h^{-3/2} e^{-b_1r^{1/(d+1)}}\norm{\widetilde{\phi}}_{H^{-1/2}(\Gamma)},
\end{align*}
where we defined $b_1 := -\frac{\ln(q)}{C_{\rm dim}^{1/(d+1)}}q^{d/(d+1)}(3+\eta)^{-d/(1+d)} > 0$ 
to obtain \linebreak $q^k = e^{-b_1r^{1/(d+1)}}$.
Therefore, we get
\begin{equation}\label{eq:Schurtemp1}
\abs{\skp{\widetilde{\phi}-\widehat{\phi},V\psi}_{L^2(B_{(1+\delta)R_{\sigma}}\cap\Gamma_{\rho})}} \lesssim 
 h^{-3/2} e^{-b_1r^{1/(d+1)}}\norm{\widetilde{\phi}}_{H^{-1/2}(\Gamma)}\norm{V\psi}_{L^2(\Gamma)}.
\end{equation}
The ellipticity of $V$, 
$\operatorname*{supp} (\widetilde{\phi} - \phi) = \operatorname*{supp} \phi_{\rho} \subset \overline{\Gamma_{\rho}}$,
 and the orthogonality \eqref{eq:SchurOrthogonality} yield
\begin{align}\label{eq:Schurtemp2}
\norm{\widetilde{\phi}-\phi}_{H^{-1/2}(\Gamma)}^2&\lesssim 
\skp{V(\widetilde{\phi}-\phi),\widetilde{\phi}-\phi}
= -\skp{V\phi,\widetilde{\phi}-\phi} \nonumber\\
 &\lesssim \norm{V\phi}_{H^{1/2}(\Gamma)}\norm{\widetilde{\phi}-\phi}_{H^{-1/2}(\Gamma)}  
\lesssim \norm{\phi}_{L^2(\Gamma)}\norm{\widetilde{\phi}-\phi}_{H^{-1/2}(\Gamma)}.
\end{align}
Thus, with the triangle inequality, \eqref{eq:Schurtemp2}, and the stability of $V:L^2(\Gamma)\rightarrow H^1(\Gamma)$, we 
can estimate \eqref{eq:Schurtemp1} by 
\begin{align*}
\abs{\skp{\widetilde{\phi}-\widehat{\phi},V\psi}_{L^2(B_{(1+\delta)R_{\sigma}}\cap\Gamma_{\rho})}} &\lesssim 
 h^{-3/2} e^{-br^{1/(d+1)}}\Big(
\norm{\widetilde{\phi}-\phi}_{H^{-1/2}(\Gamma)}\norm{V\psi}_{L^2(\Gamma)} \\
&\quad+\norm{\phi}_{H^{-1/2}(\Gamma)}\norm{V\psi}_{L^2(\Gamma)}\Big) \\
&\lesssim 
 h^{-3/2} e^{-br^{1/(d+1)}}\norm{\phi}_{L^2(\Gamma)}\norm{\psi}_{L^2(\Gamma)}. 
\end{align*}
For the second term in \eqref{eq:tmpSchur}, 
we exploit the asymptotic smoothness of the Green's function $G(\cdot,\cdot)$: 
Lemma~\ref{lem:lowrankGreensfunction}
can be applied with $B_Y = B_{R_{\sigma}}$ and $D_X = \Gamma \backslash B_{(1+\delta)R_{\sigma}}$, 
where the choice of $\delta$ implies 
\begin{equation}
\label{eq:degenerate-approximation-admissibility}
\operatorname{dist}(B_Y,D_X)\geq \frac{1}{2\sqrt{d}(1+\eta)} \operatorname*{diam}(B_Y).
\end{equation}
Therefore, we get an approximation $G_r(x,y) = \sum_{i=1}^r g_{1,i}(x) g_{2,i}(y)$ such that 
\begin{align}
\label{eq:degenerate-approximation-error}
\norm{G(x,\cdot)-{G}_r(x,\cdot)}_{L^{\infty}(B_{R_{\sigma}})} 
\!\lesssim\! \frac{1}{\operatorname{dist}(\{x\},B_{R_\sigma})^{d-2}}e^{-b_2r^{1/d}} 
\, \forall x \in \Gamma\setminus B_{(1+\delta)R_{\sigma}};
\end{align}
here, the constant $b_2>0$ depends only on $d$ and $\eta$. 
As a consequence of \eqref{eq:degenerate-approximation-admissibility} and
\eqref{eq:degenerate-approximation-error}, 
the rank-$r$ operator $V_r$ given by 
$V_r\psi(x):=\int_{B_{R_{\sigma}}\cap\Gamma}{G}_r(x,y)\psi(y)ds_y$ satisfies 
\begin{align*}
&\abs{\skp{\widetilde{\phi},(V-V_r)\psi}_{L^2(\Gamma \backslash B_{(1+\delta)R_{\sigma}})}}\\ 
&\qquad=\abs{\int_{\Gamma \backslash B_{(1+\delta)R_{\sigma}}}
\widetilde{\phi}(x)\int_{B_{R_{\sigma}}\cap\Gamma}(G(x,y)-{G}_r(x,y))\psi(y)ds_y ds_x} \\
&\qquad\lesssim 
\norm{\widetilde{\phi}}_{L^2(\Gamma)}
\sqrt{\operatorname*{meas}(\Gamma \cap B_{R_\sigma})}
\norm{G-\widetilde{G}_r}_{L^{\infty}\left((\Gamma \backslash B_{(1+\delta)R_{\sigma}}) \times (B_{R_{\sigma}}\cap\Gamma)\right)}
\norm{\psi}_{L^2(\Gamma)} \\
&\qquad\lesssim 
h^{-1/2}\delta^{2-d} R_\sigma^{(3-d)/2} e^{-b_2r^{1/d}}\norm{\widetilde{\phi}}_{H^{-1/2}(\Gamma)} \norm{\psi}_{L^2(\Gamma)} \\
&\qquad\lesssim 
h^{-1/2}e^{-b_2r^{1/d}} \norm{\phi}_{L^2(\Gamma)}
\norm{\psi}_{L^2(\Gamma)},
\end{align*}
where the last two inequalities follow from the inverse estimate Lemma~\ref{lem:inverseinequality}, 
the stability estimate \eqref{eq:Schurtemp2} for the mapping $\phi \mapsto \widetilde \phi$, the assumption 
$d \leq 3$ as well as $R_{\sigma} \leq \eta\operatorname*{diam}(\Omega)$, and the choice $\delta = \frac{1}{1+\eta}$. 
Here, the hidden constant additionally depends on $\eta$.
Therefore, we get 
\begin{align*}
&\abs{\skp{V\widetilde{\phi},\psi} - 
\skp{\widehat{\phi},V\psi}_{L^2(B_{(1+\delta)R_{\sigma}}\cap\Gamma_{\rho})} - 
\skp{\widetilde{\phi},V_r\psi}_{L^2(\Gamma \backslash B_{(1+\delta)R_{\sigma}})}} \\    
&\qquad\qquad\lesssim h^{-3/2}e^{-br^{1/(d+1)}}\norm{\phi}_{L^2(\Gamma)}
\norm{\psi}_{L^2(\Gamma)},
\end{align*}
with $b := \min\{b_1,b_2\}$.
Since the mapping 
$(\phi,\psi) \mapsto \skp{\widehat{\phi},V\psi}_{L^2(B_{(1+\delta)R_{\sigma}}\cap\Gamma_{\rho})} + 
\skp{\widetilde{\phi},V_r\psi}_{L^2(\Gamma \backslash B_{(1+\delta)R_{\sigma}})}$
defines a bounded bilinear form on $L^2(\Gamma)$,
there exists a linear operator $\widehat{V}:L^2(\Gamma)\rightarrow L^2(\Gamma)$ such that
\begin{equation*}
\skp{\widehat{\phi},V\psi}_{L^2(B_{(1+\delta)R_{\sigma}}\cap\Gamma_{\rho})} + 
\skp{\widetilde{\phi},V_r\psi}_{L^2(\Gamma \backslash B_{(1+\delta)R_{\sigma}})}
= \skp{\widehat{V}\phi,\psi},
\end{equation*}
and the dimension of the range of $\widehat{V}$ is bounded by $2r$.
Therefore, we get a
matrix $\mathbf{S}_{r}(\tau,\sigma)$ of rank $2r$ such that
\begin{align*}
\norm{\mathbf{S}(\tau,\sigma)-\mathbf{S}_{r}(\tau,\sigma)}_2 &= 
\sup_{\boldsymbol{\phi}\in\mathbb{R}^{\abs{\tau}},\boldsymbol{\psi}\in \mathbb{R}^{\abs{\sigma}}} 
\frac{\abs{\boldsymbol{\phi}^T(\mathbf{S}(\tau,\sigma)-\mathbf{S}_{r}(\tau,\sigma))
\boldsymbol{\psi}}}{\norm{\boldsymbol{\phi}}_2\norm{\boldsymbol{\psi}}_2} \\
&\leq C h^{d-3/2}e^{-br^{1/(d+1)}},
\end{align*}
where we have used \eqref{eq:isomorphism}.
 The estimate $\frac{1}{\norm{\mathbf{V}}_2}\lesssim h^{-d}$ from \cite[Lemma~12.6]{Steinbach}
and $h \simeq N^{-1/(d-1)}$ finish the proof.
\end{proof}

As a direct consequence of the representation \eqref{eq:SchurRepresentation} and the results from 
Section~\ref{sec:Approximation-solution}, we can get a blockwise rank-$r$ approximation 
of the inverse of the Schur complement
$\mathbf{S}(\tau,\tau)$. For the existence of the inverse $\mathbf{S}(\tau,\tau)^{-1}$, 
we refer to the next subsection. 
For a given right-hand side $f \in L^2(\Gamma)$, \eqref{eq:SchurRepresentation} implies that
solving $\mathbf{S}(\tau,\tau)\boldsymbol{\phi} = \mathbf{f}$
with $\mathbf{f} \in\mathbb{R}^{\abs{\tau}}$ defined by $\mathbf{f}_i = \skp{f,\chi_{i_{\tau}}}$,
is equivalent to solving 
$a(\widetilde{\phi},\psi) = \skp{f,\psi}$ 
for all $\psi \in S^{0,0}(\mathcal{T}_h)$  with $\operatorname*{supp} \psi \subset \overline{\Gamma_{\tau}}$. 
Let $\tau_1\times\sigma_1 \subset \tau\times\tau$ be an $\eta$-admissible subblock.
For $f \in L^2(\Gamma)$ with $\operatorname*{supp} f \subset B_{R_{\sigma_1}}\cap\Gamma$, the support properties
as well as the admissibility condition \eqref{eq:admissibility} 
for the cluster pair $(\tau_1,\sigma_1)$ imply the orthogonality
\begin{equation*}
a(\widetilde{\phi},\psi) = 0 \quad \forall \psi \in S^{0,0}(\mathcal{T}_h) \, \text{with} \, 
\operatorname*{supp} \psi \subset B_{R_{\tau_1}}\cap\overline{\Gamma_{\tau}}.
\end{equation*}
Therefore, we have $\widetilde{V}\widetilde{\phi} \in \mathcal{H}_{h,0}(B_{R_{\tau_1}},\Gamma_{\tau})$, and
Lemma~\ref{cor:lowdimapp} provides an approximation to $\widetilde{\phi}$
on $B_{R_{\tau_1}}\cap\Gamma_{\tau}$. Then, a rank-$r$ factorization of the subblock 
$\mathbf{S}(\tau,\tau)^{-1}|_{\tau_1\times\sigma_1}$ can be constructed as in Section~\ref{sec:H-matrix-approximation}, 
which is summarized in the following theorem.

\begin{theorem}
Let $\tau \subset \mathcal{I}$, $\rho := \{i\in \mathcal{I} : i < \min(\tau)\}$,
$\tau_1\times\sigma_1 \subset \tau\times\tau$ be $\eta$-admissible, and let
the Schur complement $\mathbf{S}(\tau,\tau)$ be defined in \eqref{eq:defSchur}.
Then, there exist rank-$r$ matrices 
$\mathbf{X}_{\tau_1\sigma_1} \in \mathbb{R}^{\abs{\tau_1}\times r}$, $\mathbf{Y}_{\tau_1\sigma_1} \in \mathbb{R}^{\abs{\sigma_1}\times r}$
such that 
\begin{equation}
\norm{\mathbf{S}(\tau,\tau)^{-1}|_{\tau_1\times\sigma_1} - \mathbf{X}_{\tau_1\sigma_1}\mathbf{Y}_{\tau_1\sigma_1}^T}_2 
\leq C_{\rm apx} N^{(d+2)/(d-1)} e^{-br^{1/(d+1)}}.
\end{equation}
The constants $C_{\rm apx}$ depends only on  
$\Omega$, $d$, and the $\gamma$-shape regularity of the quasiuniform triangulation $\mathcal{T}_h$,
and the constant $b>0$ additionally depends on $\eta$.
\end{theorem}

\subsection{Existence of $\mathcal{H}$-Cholesky decomposition: conclusion of the proof of Theorem~\ref{th:HLU}}

In this subsection, we will use the approximation of the Schur complement from the previous section
to prove the existence of an (approximate) $\mathcal{H}$-Cholesky decomposition. 
We start with a hierarchical relation of the Schur complements $\mathbf{S}(\tau,\tau)$. \newline 

The Schur complements $\mathbf{S}(\tau,\tau)$ for a block $\tau \in \mathbb{T}_{\mathcal{I}}$ 
can be derived from the Schur complements of its sons $\tau_1$, $\tau_2$ by 
\begin{equation*}
\mathbf{S}(\tau,\tau) = \begin{pmatrix} \mathbf{S}(\tau_1,\tau_1) & \mathbf{S}(\tau_1,\tau_2) \\ 
\mathbf{S}(\tau_2,\tau_1) & \mathbf{S}(\tau_2,\tau_2) + \mathbf{S}(\tau_2,\tau_1)\mathbf{S}(\tau_1,\tau_1)^{-1}\mathbf{S}(\tau_1,\tau_2) \end{pmatrix},
\end{equation*}
A proof of this relation can be found in \cite[Lemma 3.1]{Bebendorf07}. One should note that
the proof does not use any properties of the matrix $\mathbf{V}$ other than invertibility and 
existence of a Cholesky decomposition. 
Moreover, we have by definition of $\mathbf{S}(\tau,\tau)$ that $\mathbf{S}(\mathcal{I},\mathcal{I}) = \mathbf{V}$.

If $\tau$ is a leaf, we get the Cholesky decomposition of $\mathbf{S}(\tau,\tau)$ by the classical 
Cholesky decomposition,
which exists since $\mathbf{V}$ has a Cholesky decomposition.
If $\tau$ is not a leaf, we use the hierarchical relation of the Schur complements to
define a Cholesky decomposition of the Schur complement
 $\mathbf{S}(\tau,\tau)$ by 
\begin{equation}\label{eq:LUdefinition}
\mathbf{C}(\tau) := \begin{pmatrix} \mathbf{C}(\tau_1) & 0 \\ \mathbf{S}(\tau_2,\tau_1)(\mathbf{C}(\tau_1)^T)^{-1} & \mathbf{C}(\tau_2)  \end{pmatrix}, \quad 
\end{equation}
with $\mathbf{S}(\tau_1,\tau_1) = \mathbf{C}(\tau_1)\mathbf{C}(\tau_1)^T$,
 $\mathbf{S}(\tau_2,\tau_2) = \mathbf{C}(\tau_2)\mathbf{C}(\tau_2)^T$ and indeed get 
$\mathbf{S}(\tau,\tau) = \mathbf{C}(\tau) \mathbf{C}(\tau)^T$.
Moreover, the uniqueness of the Cholesky decomposition of $\mathbf{V}$ implies that due to
$\mathbf{C}\mathbf{C}^T = \mathbf{V} = \mathbf{S}(\mathcal{I},\mathcal{I}) = \mathbf{C}(\mathcal{I})\mathbf{C}(\mathcal{I})^T$, we have
$\mathbf{C} = \mathbf{C}(\mathcal{I})$.

The existence of the inverse $\mathbf{C}(\tau_1)^{-1}$
follows from the representation \eqref{eq:LUdefinition}
by induction over the levels, since on a leaf the existence is clear and the matrices 
$\mathbf{C}(\tau)$ are block triangular matrices. Consequently, the inverse of
$\mathbf{S}(\tau,\tau)$ exists. 

Moreover, as shown in \cite[Lemma~{22}]{GrasedyckKriemannLeBorne}  in the context of $LU$-factorizations
instead of Cholesky decompositions, the restriction of the lower triangular part \linebreak
$\mathbf{S}(\tau_2,\tau_1)(\mathbf{C}(\tau_1)^T)^{-1}$ 
of the matrix $\mathbf{C}(\tau)$ to a subblock $\tau_2'\times\tau_1'$
with $\tau_i'$ a son of $\tau_i$ satisfies
\begin{equation}
\label{eq:foo}
\left(\mathbf{S}(\tau_2,\tau_1)(\mathbf{C}(\tau_1)^T)^{-1}\right)|_{\tau_2'\times\tau_1'} = 
\mathbf{S}(\tau_2',\tau_1')(\mathbf{C}(\tau_1')^T)^{-1}.
\end{equation}

The following lemma shows that the spectral norm of the inverse
$\mathbf{C}(\tau)^{-1}$ can be bounded by the norm of the inverse
$\mathbf{C}(\mathcal{I})^{-1}$.

\begin{lemma}\label{lem:LUnorm}
For $\tau\in \mathbb{T}_{\mathcal{I}}$,
let $\mathbf{C}(\tau)$ be given by \eqref{eq:LUdefinition}. Then,
\begin{align*}
\max_{\tau\in\mathbb{T}_{\mathcal{I}}}\norm{\mathbf{C}(\tau)^{-1}}_2 = \norm{\mathbf{C}(\mathcal{I})^{-1}}_2. 
\end{align*}
\end{lemma}
\begin{proof}
With the block structure of \eqref{eq:LUdefinition}, we get the inverse
\begin{equation*}
\mathbf{C}(\tau)^{-1} = \begin{pmatrix} \mathbf{C}(\tau_1)^{-1} & 0 \\ 
-\mathbf{C}(\tau_2)^{-1} \mathbf{S}(\tau_2,\tau_1)(\mathbf{C}(\tau_1)^T)^{-1}\mathbf{C}(\tau_1)^{-1} & 
\mathbf{C}(\tau_2)^{-1} \end{pmatrix}.
\end{equation*}
So, we get by choosing $\mathbf{x}$ such that $\mathbf{x}_i = 0$ for $i \in \tau_1$ that
\begin{align*}
\norm{\mathbf{C}(\tau)^{-1}}_2 &= \sup_{\mathbf{x}\in \mathbb{R}^{\abs{\tau}},\norm{x}_2=1}
\norm{\mathbf{C}(\tau)^{-1}\mathbf{x}}_2 \\ 
&\geq \sup_{\mathbf{x}\in \mathbb{R}^{\abs{\tau_2}},\norm{x}_2=1}\norm{\mathbf{C}(\tau_2)^{-1}\mathbf{x}}_2 
= \norm{\mathbf{C}(\tau_2)^{-1}}_2.
\end{align*}
The same argument for $\left(\mathbf{C}(\tau)^{-1}\right)^T$ leads to 
\begin{align*}
\norm{\mathbf{C}(\tau)^{-1}}_2 = \norm{\left(\mathbf{C}(\tau)^{-1}\right)^T}_2 \geq \norm{\mathbf{C}(\tau_1)^{-1}}_2.
\end{align*}
Thus, we have $\norm{\mathbf{C}(\tau)^{-1}}_2 \geq \max_{i=1,2}\norm{\mathbf{C}(\tau_i)^{-1}}_2$ 
and as a consequence \linebreak
$\max_{\tau\in\mathbb{T}_{\mathcal{I}}}\norm{\mathbf{C}(\tau)^{-1}}_2 = \norm{\mathbf{C}(\mathcal{I})^{-1}}_2$.
\end{proof}

We are now in position to prove Theorem~\ref{th:HLU}:

\begin{proof}[Proof of Theorem~\ref{th:HLU}]
In the following, we show that every admissible subblock $\tau\times\sigma$ of $\mathbf{C}(\mathcal{I})$, 
recursively defined by \eqref{eq:LUdefinition}, has a rank-$r$ approximation. Since an admissible block
of the lower triangular part of $\mathbf{C}(\mathcal{I})$
has to be a subblock of a matrix $\mathbf{C}(\tau')$ for some 
$\tau' \in \mathbb{T}_{\mathcal{I}}$, we get in view of \eqref{eq:foo} that 
$\mathbf{C}(\mathcal{I})|_{\tau\times\sigma} = \mathbf{S}(\tau,\sigma)(\mathbf{C}(\sigma)^T)^{-1}$.
Theorem~\ref{lem:Schur} provides a rank-$r$ approximation 
${\mathbf S}_{r}(\tau,\sigma)$ to ${\mathbf S}(\tau,\sigma)$. Therefore, we can estimate 
\begin{align*}
\norm{\mathbf{C}(\mathcal{I})|_{\tau\times\sigma} - \mathbf{S}_{r}(\tau,\sigma)(\mathbf{C}(\sigma)^T)^{-1}}_2 
&=\norm{\left(\mathbf{S}(\tau,\sigma)-\mathbf{S}_{r}(\tau,\sigma)\right)(\mathbf{C}(\sigma)^T)^{-1}}_2 \\
&\leq C_{\rm sc}  N^{3/(2d-2)}  e^{-br^{1/(d+1)}}\norm{(\mathbf{C}(\sigma)^T)^{-1}}_2\norm{\mathbf{V}}_2.
\end{align*} 
Since $\mathbf{S}_{r}(\tau,\sigma)(\mathbf{C}(\sigma)^T)^{-1}$ is a rank-$r$ matrix for each $\eta$-admissible 
cluster pair $(\tau,\sigma)$, we immediately get an $\mathcal{H}$-matrix approximation $\mathbf{C}_{\mathcal{H}}$
of the Cholesky factor $\mathbf{C}(\mathcal{I}) = \mathbf{C}$. With Lemma~\ref{lem:spectralnorm} 
and Lemma~\ref{lem:LUnorm}, we get
\begin{equation*}
\norm{\mathbf{C}-\mathbf{C}_{\mathcal{H}}}_2\leq C_{\rm sc} C_{\rm sp} N^{3/(2d-2)}\operatorname{depth}(\mathbb{T}_{\mathcal{I}})e^{-br^{1/(d+1)}}
\norm{\mathbf{C}^{-1}}_2\norm{\mathbf{V}}_2,
\end{equation*}
and with $\norm{\mathbf{V}}_2 = \norm{\mathbf{C}}_2^2$, we conclude the proof of (i). 

Since $\mathbf{V}=\mathbf{C}\mathbf{C}^T$, the triangle inequality finally leads to
\begin{align*}
\norm{\mathbf{V}-\mathbf{C}_{\mathcal{H}}\mathbf{C}_{\mathcal{H}}^T}_2 &\leq 
\norm{\mathbf{C}-\mathbf{C}_{\mathcal{H}}}_2\norm{\mathbf{C}^T}_2 + 
\norm{\mathbf{C}^T-\mathbf{C}_{\mathcal{H}}^T}_2 \norm{\mathbf{C}}_2 \\
& \phantom{\leq}\, + 
\norm{\mathbf{C}-\mathbf{C}_{\mathcal{H}}}_2\norm{\mathbf{C}^T-\mathbf{C}_{\mathcal{H}}^T}_2 \\
&\leq 2C_{\rm sc}C_{\rm sp}\kappa_2(\mathbf{C})
\operatorname{depth}(\mathbb{T}_{\mathcal{I}})N^{3/(2d-2)} e^{-br^{1/(d+1)}}\norm{\mathbf{V}}_2\\
 &\phantom{\leq}\, +\kappa_2(\mathbf{C})^2C_{\rm sc}^2C_{\rm sp}^2 
\operatorname{depth}(\mathbb{T}_{\mathcal{I}})^2N^{3/(d-1)} e^{-2br^{1/(d+1)}}
\frac{\norm{\mathbf{V}}_2^2}{\norm{\mathbf{C}}_2^2},
\end{align*}
and the equality $\kappa_2(\mathbf{V}) = \kappa_2(\mathbf{C})^2$ finishes the proof.
\end{proof}

\section{Extensions}
\label{sec:extensions}
\subsection{The Poincar\'{e}-Steklov operator}

The interior Poincar\'{e}-Steklov operator $S^{\rm int}$ is defined as 
$S^{\rm int}:=V^{-1}\left(\frac{1}{2}I+K\right) : H^{1/2}(\Gamma) \rightarrow H^{-1/2}(\Gamma)$, 
where $K$ denotes the double-layer operator. In a similar way, the 
exterior Poincar\'{e}-Steklov operator $S^{\rm ext}$ is given by 
$S^{\rm ext}=  - V^{-1} (\frac{1}{2}I - K)$. 

The discrete Poincar\'{e}-Steklov operators are given by 
$\mathbf{S}^{\rm int} = \mathbf{V}^{-1}\left(\frac{1}{2}\mathbf{M}+\mathbf{K}\right)$ and 
$\mathbf{S}^{\rm ext} = -\mathbf{V}^{-1}\left(\frac{1}{2}\mathbf{M}-\mathbf{K}\right)$, 
where $\mathbf{K}$ is the stiffness matrices corresponding to $K$, 
and $\mathbf{M}$ is the mass matrix. Here, we consider piecewise affine basis functions 
for the the discretizations $\mathbf{K},\mathbf{M}$.
We illustrate the use of the $\mathcal{H}$-arithmetic to 
derive $\mathcal{H}$-matrix approximations
to the discrete Poincar\'{e}-Steklov operators, which is stated in the following corollary
of Theorem~\ref{th:Happrox}.
\begin{corollary}
Fix $\eta>0$.   Let a partition $P$ of $\mathcal{I}\times \mathcal{I}$ be based on a cluster tree 
$\mathbb{T}_{\mathcal{I}}$ created by the geometric clustering algorithm from \cite[Section 5.4.2]{HackbuschBuch}.
Let $\mathbf{S} \in \{{\mathbf S}^{\rm int}, {\mathbf S}^{\rm ext}\}$. 
Then, there is a blockwise rank-$r$ matrix $\mathbf{S}_{\mathcal{H}}$ such that
\begin{equation}\label{eq:PSresult}
\norm{\mathbf{S}-\mathbf{S}_{\mathcal{H}}}_2 \leq C_{\rm PS} N^{(d+2)/(d-1)} \log N 
\exp\left(-b\left(\frac{r}{\log N +1}\right)^{1/(d+1)}\right).
\end{equation}
The constant $C_{\rm PS}>0$ depends only on $\Omega,d$, 
and the $\gamma$-shape regularity of the quasiuniform triangulation $\mathcal{T}_h$, and the 
constant $b>0$ depends additionally on $\eta$.
\end{corollary}
\begin{proof}
By $\mathcal{H}$-arithmetic, \cite[Theorem 2.24]{GrasedyckHackbusch}, the rank of the multiplication
of $\mathcal{H}$-matrices increases
by a factor of $C_{\rm id}C_{\rm sp}(\operatorname{depth}(\mathbb{T}_{\mathcal{I}})+1)$, where the appearing idempotency constant 
$C_{\rm id}$ is defined in \cite{GrasedyckHackbusch} and can be bounded uniformly in $N$ for 
geometrically balanced cluster trees. 
\end{proof}

\begin{remark}
In computations, due to stability reasons, usually the symmetric formulation of the Poincar\'{e}-Steklov operator
$S^{\rm int}:=W+\left(\frac{1}{2}I+K'\right)V^{-1}\left(\frac{1}{2}I+K\right)$ with the hypersingular integral operator
$W:H^{1/2}(\Gamma)\rightarrow H^{-1/2}(\Gamma)$ and the adjoint double-layer potential $K'$ is used, see, e.g., \cite{Steinbach}. 
Using $\mathcal{H}$-arithmetics, an approximation
for this representation can be derived as well, but leads to an additional logarithmic factor in the exponential in
\eqref{eq:PSresult}. Since we are only interested in an existence result, the non symmetric formulation is sufficient
for our purpose.
\end{remark}

\subsection{$\mathcal{H}^2$-approximation}

In this section, we briefly describe how an approximation in the more refined framework of $\mathcal{H}^2$-matrices can be derived.
The main advantage of $\mathcal{H}^2$-matrices compared to $\mathcal{H}$-matrices is that the storage complexity
as well as the complexity of the matrix-vector multiplication is 
$\mathcal{O}(rN)$, i.e., linear in the degrees of freedom
instead of the logarithmic-linear complexity of $\mathcal{O}(rN \log N)$ for $\mathcal{H}$-matrices.

In fact, \cite{Boerm} proves that an $\mathcal{H}^2$-approximation can be derived from 
blockwise estimates as in Theorem~\ref{th:blockapprox}. 
Therefore, this can be done in the same way, and this section follows the lines of \cite{Boerm}.

$\mathcal{H}^2$-matrices are based on nested cluster bases, which are defined in the following.

\begin{definition}[Nested cluster basis]
A family of matrices $(\mathbf{U}_{\tau})_{\tau \in \mathbb{T}_I}, \mathbf{U}_{\tau} \in \mathbb{R}^{\abs{\tau}\times r}$ is said to be 
a {\rm nested cluster basis}, if there exists a family of {\rm transfer matrices} 
$(\mathbf{T}_{\tau})_{\tau \in \mathbb{T}_{\mathcal{I}}}, \mathbf{T}_{\tau}\in \mathbb{R}^{r\times r}$ such that
$
\mathbf{U}_{\tau}|_{\tau'} = \mathbf{U}_{\tau'}\mathbf{T}_{\tau'} \;
\forall \tau \in \mathbb{T}_{\mathcal{I}}, \tau' \in {\rm sons}(\tau).
$
If the matrices $\mathbf{U}_{\tau}$ are orthogonal for all $\tau \in \mathbb{T}_{\mathcal{I}}$, 
$(\mathbf{U}_{\tau})_{\tau \in \mathbb{T}_{\mathcal{I}}}$ is said to be a 
\rm{nested orthogonal cluster basis}.
\end{definition}

\begin{definition}[Descendants, predecessors, block row]
The {\rm set of descendants} of $\tau \in \mathbb{T}_{\mathcal{I}}$ is recursively defined by 
\begin{equation*}
\operatorname{sons}^*(\tau) := \left\{
\begin{array}{l}
 \{\tau\} \hspace{40mm}\,\textrm{if}\hspace{2mm} \operatorname{sons}(\tau) = \emptyset, \\
 \{\tau\}\cup \bigcup_{\tau'\in\operatorname{sons}(\tau)}\operatorname{sons}^*(\tau') \quad \textrm{otherwise}.
 \end{array}
 \right.
\end{equation*}
 The {\rm set of predecessors} is given by
\begin{equation*}
\rm{pred}(\tau) := \{\tau^+ \in \mathbb{T}_{\mathcal{I}}\; :\; \tau \in \rm{sons}^*(\tau^+)\}.
\end{equation*}
Further, the {\rm block row} is defined by
\begin{equation*}
\rm{row}^*(\tau) := \{\sigma \in \mathbb{T}_{\mathcal{I}}\; :\; \exists \tau^+ \in 
\rm{pred}(\tau) : (\tau^+,\sigma) \in P_{\rm far} \}. 
\end{equation*}
\end{definition}

\begin{definition}[$\mathcal{H}^2$-matrix]
Let the partition $P$ of $\mathcal{I}\times \mathcal{I}$ be based on 
the cluster tree $\mathbb{T}_{\mathcal{I}}$ and $\eta > 0$. 
Let $\left(\mathbf{S_{\tau}}\right)_{\tau \in\mathbb{T}_{\mathcal{I}}}$ and $\left(\mathbf{U_{\sigma}}\right)_{\sigma \in\mathbb{T}_{\mathcal{I}}}$
be nested cluster bases. A matrix $\mathbf{W}_{\mathcal{H}^2}$ is said to be an {\rm $\mathcal{H}^2$-matrix}, if 
for each $\eta$-admissible cluster pair $(\tau,\sigma) \in P_{\rm far}$, there is a coupling matrix 
$\mathbf{M}_{\tau\sigma} \in \mathbb{R}^{r\times r}$ such that  
$\mathbf{W}_{\mathcal{H}^2}|_{\tau\times \sigma} = \mathbf{S}_{\tau}\mathbf{M}_{\tau\sigma}\mathbf{U}_{\sigma}^T$.
\end{definition}

We refer to \cite{Boerm02} for the storage complexity of $\mathcal{O}(rN)$ for the family of transfer matrices 
$(\mathbf{T}_{\tau})_{\tau \in \mathbb{T}_{\mathcal{I}}}$ and coupling matrices $\left(\mathbf{M}_{\tau\sigma}\right)_{(\tau,\sigma) \in P_{\rm far}}$.
Since the cluster basis only needs to be stored for the leaf clusters, we get a storage requirement of $\mathcal{O}(rN)$
for the cluster bases $\left(\mathbf{S_{\tau}}\right)_{\tau \in\mathbb{T}_{\mathcal{I}}}$ 
and $\left(\mathbf{U_{\tau}}\right)_{\tau \in\mathbb{T}_{\mathcal{I}}}$,
and therefore a total storage requirement of $\mathcal{O}(rN)$ for $\mathcal{H}^2$-matrices.\\

The following theorem shows that the matrix $\mathbf{V}^{-1}$
can be approximated by an $\mathcal{H}^2$-matrix and that the error converges exponentially in the block rank.

\begin{theorem}\label{th:H2approx}
Fix the admissibility parameter $\eta>0$. Let a partition $P$ of $\mathcal{I}\times \mathcal{I}$ be based on a cluster tree 
$\mathbb{T}_{\mathcal{I}}$.
Then, there is a blockwise rank-$r$ $\mathcal{H}^2$-matrix $\mathbf{W}_{\mathcal{H}^2}$ such that
\begin{equation*}
\norm{\mathbf{V}^{-1} - \mathbf{W}_{\mathcal{H}^2}}_2 \leq C_{\rm H2} N^{(d+2)/(d-1)}{\rm depth}(\mathbb{T}_{\mathcal{I}})
\sqrt{\abs{\mathbb{T}_{\mathcal{I}}}} e^{-br^{1/(d+1)}}. 
\end{equation*}
The constant $C_{\rm H2}>0$ depends only on $\Omega,d$, 
and the $\gamma$-shape regularity of the quasiuniform triangulation $\mathcal{T}_h$, where the 
constant $b>0$ additionally depends on $\eta$.
\end{theorem}
\begin{proof}
We define the total cluster basis $M_{\tau} := \bigcup\{\sigma\;:\;\sigma \in \rm{row}^*(\tau)\}$. 
By definition of $\rm{row}^*(\tau)$, there exists a cluster $\tau^+ \in \rm{pred}(\tau)$ 
such that $(\tau^+,\sigma) \in P_{\rm far}$ for all $\sigma \in \rm{row}^*(\tau)$.
Let $B_{R_{\tau}},B_{R_{\tau^+}},B_{R_{\sigma}}$ be bounding boxes for the clusters $\tau,\tau^+,\sigma$,
and we assume that $B_{R_{\tau}} \subset B_{R_{\tau^+}}$ for $\tau \in \operatorname{sons}^*(\tau^+)$.
Then, we have
\begin{align*}
\operatorname*{diam}(B_{R_{\tau}}) &\leq \operatorname*{diam}(B_{R_{\tau^+}}) \leq 
\eta \operatorname{dist}(B_{R_{\tau^+}},B_{R_{\sigma}}) \\
&\leq \eta \operatorname{dist}(B_{R_{\tau}},B_{R_{\sigma}}) \qquad \forall \sigma \in \rm{row}^*(\tau). 
\end{align*}
Therefore, Theorem \ref{th:blockapprox} can be used to derive a low-rank approximation 
of the matrix block $\mathbf{V}^{-1}|_{\tau\times M_{\tau}}$, i.e., it provides matrices 
$\mathbf{X}_{\tau M_{\tau}}\in\mathbb{R}^{\abs{\tau}\times r}$,$\mathbf{Y}_{\tau M_{\tau}} \in \mathbb{R}^{\abs{M_{\tau}}\times r}$
with $r = C_{\rm dim}(2+\eta)^dq^{-d}k^{d+1}$, such that
\begin{equation*}
\norm{\mathbf{V}^{-1}|_{\tau\times M_{\tau}}-\mathbf{X}_{\tau M_{\tau}}\mathbf{Y}_{\tau M_{\tau}}^T}_2 
\lesssim N^{(d+2)/(d-1)} q^k \simeq N^{(d+2)/(d-1)}e^{-br^{1/(d+1)}},
\end{equation*}
where the constant $b>0$ depends only on $\Omega,d$, the $\gamma$-shape regularity of 
the quasiuniform triangulation $\mathcal{T}_h$ and $\eta$.
Then, \cite[Corollary 6.18]{BoermBuch} states that there exists a nested orthogonal cluster basis 
$(\mathbf{U}_{\tau})_{\tau \in \mathbb{T}_{\mathcal{I}}} \in \mathbb{R}^{\abs{\tau}\times r}$ of 
rank $r \leq C_{\rm dim}\left(2+\eta\right)^dq^{-d}k^{d+1}$  such that for each $(\tau,\sigma) \in P_{\rm far}$ we have
\begin{align*}
\norm{\mathbf{V}^{-1}|_{\tau\times\sigma}-\mathbf{U}_{\tau}\mathbf{U}_{\tau}^T\mathbf{V}^{-1}|_{\tau\times\sigma}}_2^2 
&\lesssim 
\sum_{\tau \in \mathbb{T}_{\mathcal{I}}}\norm{\mathbf{V}^{-1}|_{\tau\times M_{\tau}}-
\mathbf{X}_{\tau M_{\tau}}\mathbf{Y}_{\tau M_{\tau}}^T}_2^2 \\
&\lesssim N^{(2d+4)/(d-1)}\abs{\mathbb{T}_{\mathcal{I}}}e^{-2br^{1/(d+1)}}.
\end{align*}
With $\mathbf{M}_{\tau\sigma} := \mathbf{U}_{\tau}^T\mathbf{V}^{-1}|_{\tau\times\sigma}\mathbf{U}_{\sigma}$, the matrix 
\begin{equation}\label{eq:Hsquare}
\mathbf{W}_{\mathcal{H}^2} = \left\{
\begin{array}{l}
 \mathbf{U}_{\tau}\mathbf{M}_{\tau\sigma}\mathbf{U}_{\sigma}^T, \quad\textrm{for}\,\, (\tau,\sigma) \in P_{\text{far}}, \\
 \mathbf{V}^{-1}|_{\tau \times \sigma}, \, \qquad \textrm{otherwise}
 \end{array}
 \right.
\end{equation}
is then the desired $\mathcal{H}^2$-matrix approximating $\mathbf{V}^{-1}$. The symmetry of $\mathbf{V}^{-1}$, 
i.e. $\mathbf{V}^{-1} = \mathbf{V}^{-T}$, and $\norm{\mathbf{A}^T}_2 = \norm{\mathbf{A}}_2$ 
for $\mathbf{A} \in \mathbb{R}^{\abs{\tau}\times \abs{\sigma}}$ imply
\begin{align*}
\norm{\mathbf{V}^{-1}|_{\tau\times\sigma}-\mathbf{W}_{\mathcal{H}^2}|_{\tau\times\sigma}}_2 &\leq
\norm{\mathbf{V}^{-1}|_{\tau\times\sigma}-\mathbf{U}_{\tau}\mathbf{U}_{\tau}^T 
\mathbf{V}^{-1}|_{\tau\times\sigma}}_2 \\ 
& \phantom{\leq}\, + \norm{\mathbf{U}_{\tau}\mathbf{U}_{\tau}^T(\mathbf{V}^{-1}|_{\tau\times\sigma}-
\mathbf{V}^{-1}|_{\tau\times\sigma}\mathbf{U}_{\sigma}\mathbf{U}_{\sigma}^T)}_2 \\
&\leq\norm{\mathbf{V}^{-1}|_{\tau\times\sigma}-\mathbf{U}_{\tau}\mathbf{U}_{\tau}^T 
\mathbf{V}^{-1}|_{\tau\times\sigma}}_2 \\
& \phantom{\leq}\, +
 \norm{\mathbf{V}^{-1}|_{\sigma\times\tau}-\mathbf{U}_{\sigma}\mathbf{U}_{\sigma}^T 
\mathbf{V}^{-1}|_{\sigma \times\tau}}_2 \\
&\lesssim N^{(d+2)/(d-1)}\sqrt{\abs{\mathbb{T}_{\mathcal{I}}}}e^{-br^{1/(d+1)}}.
\end{align*}
Finally, Lemma~\ref{lem:spectralnorm} finishes the proof.
\end{proof}

For a quasiuniform mesh and typical clustering strategies, there holds $\abs{\mathbb{T}_{\mathcal{I}}} \simeq N$ 
and $\operatorname{depth}(\mathbb{T}_{\mathcal{I}}) \simeq \log N$.
Therefore, the rank $r$ of the $\mathcal{H}^2$-matrix needed to obtain an accuracy 
$\varepsilon$ in Theorem~\ref{th:H2approx}
is given by $r \simeq \abs{\ln\varepsilon-(\frac{d+2}{d-1}+\frac{1}{2})\log N - \log\log N}^{d+1} 
\simeq \left(C_1+C_2\log N\right)^{d+1}$. 
In comparison, for the $\mathcal{H}$-matrix approximation in Theorem~\ref{th:Happrox} a smaller rank of
$r \!\simeq\! \abs{\ln\varepsilon-\frac{d+2}{d-1}\log N - \log\log N}^{d+1}\! \simeq\! \left(C_1+C_3\log N\right)^{d+1}$
is sufficient. However, since the storage requirement for $\mathcal{H}^2$-matrices is given by $\mathcal{O}(rN)$ in comparison
to the storage requirement for $\mathcal{H}$-matrices of $\mathcal{O}(rN\log N)$, we observe that in terms of storage,
the $\mathcal{H}^2$-approximation leads to better results.

We refer to \cite{Boerm} for numerical examples concerning this comparison for FEM matrices.

\section{Numerical Examples}
In this section, we present some numerical examples in two and three dimensions to illustrate
our theoretical estimates derived in the previous sections.
Further numerical examples for the $\mathcal{H}$-matrix approximation of inverse BEM matrices and black-box
preconditioning with an $\mathcal{H}$-LU decomposition can be found in 
\cite{GrasedyckDissertation,Bebendorf05,Grasedyck05,BoermBuch}.

With the choice $\eta=2$ for the admissibility parameter in \eqref{eq:admissibility},
the clustering is done by the standard geometric clustering algorithm,
i.e., by choosing axis parallel bounding boxes of minimal volume and
splitting these bounding boxes in half across the largest face 
until they are admissible or contain less degrees of freedom than $n_{\text{leaf}}$,
which we choose as $n_{\text{leaf}} = 25$ for our computations.
An approximation $\mathbf{W}_{\mathcal{H}}$ to the inverse Galerkin matrix $\mathbf{V}^{-1}$ is computed by 
using a truncated singular value decomposition of the exact inverse.
Throughout, we use the C-library HLib \cite{HLib}.\\ 

\subsection{2D-Example}

We consider the L-shaped domain 
$\Omega = (0,1)\times(0,\frac{1}{2}) \cup (0,\frac{1}{2})\times[\frac{1}{2},1)$ 
and the lowest-order discretization of the simple-layer potential 
\begin{equation*}
\mathbf{V}_{jk} = \skp{V\chi_k,\chi_j}
\end{equation*}
from \eqref{eq:Galerkin}.

In Figure \ref{fig:2DBEM}, we compare the decrease of the upper bound 
$\norm{\mathbf{I}-\mathbf{V}\mathbf{W}_{\mathcal{H}}}_2$ of the relative error with the increase 
in the block-rank for a fixed number $N = 16,384$ of degrees of freedom, where 
the largest block of $\mathbf{W}_{\mathcal{H}}$ has a size of $2,048^2$. Moreover,
Table \ref{tab:2DBEM} shows the storage requirement for the computed $\mathcal{H}$-matrix approximation 
and the percentage of memory needed compared to the full representation.
As theoretically expected, we observe a linear growth in the rank $r$ for the storage requirements.
Moreover, we remark that $\norm{\mathbf{W}_{\mathcal{H}}}_2 = 1.22\cdot10^{8}$.

\begin{figure}[hbt]
\begin{minipage}{.50\linewidth}
\centering
\psfrag{Error}[c][c]{%
 \footnotesize  Error}
\psfrag{Block rank r }[c][c]{%
 \footnotesize  Block rank r}
\psfrag{asd}[l][l]{\footnotesize $\exp(-2.4\, r)$}
\psfrag{jkl}[l][l]{\footnotesize $\norm{I-VW_{\mathcal{H}}}_2$}
\includegraphics[width=0.90\textwidth]{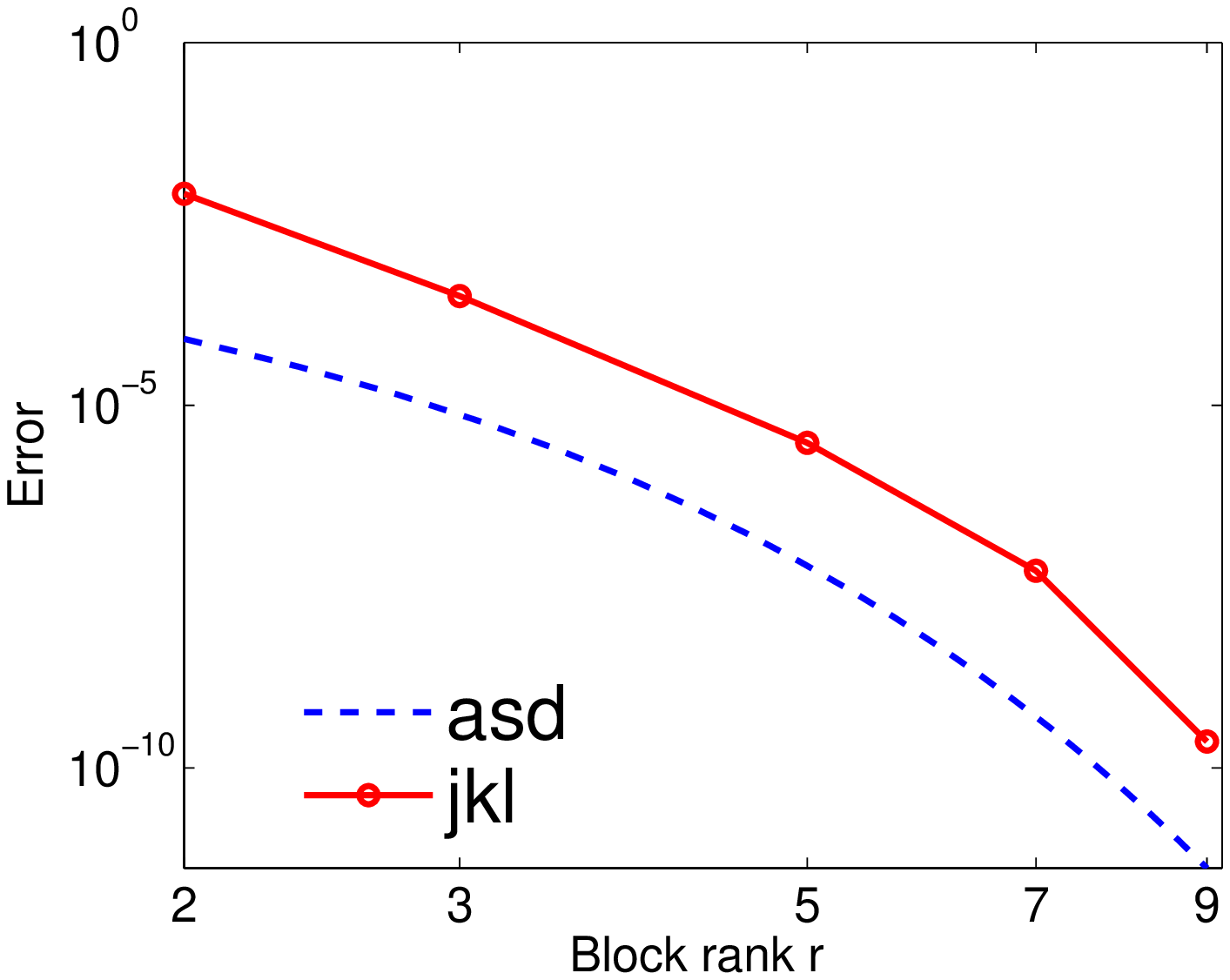}
\caption{\footnotesize Exponential convergence in block rank}
\label{fig:2DBEM}
\end{minipage}
\begin{minipage}{.49\linewidth}
\vspace{10mm}
\centering
\begin{tabular}{|c|c|c|}
\hline & Storage & Compressed \\
 $r$ & $\mathbf{W}_{\mathcal{H}}$ (MB) & to (\%)\\ 
\hline 2 & 24.5 & 1.2 \\ 
\hline 3 & 32.6 & 1.6 \\ 
\hline 5 & 49.0 & 2.4 \\ 
\hline 7 & 65.3 & 3.2 \\ 
\hline 9 & 81.6 & 4.0 \\ 
\hline
\end{tabular}
\vspace{6mm}
\captionof{table}{\footnotesize Storage (MB) for $\mathbf{W}_{\mathcal{H}}$}
\label{tab:2DBEM} 
\end{minipage}
\end{figure}

We observe exponential convergence in the block rank, where the convergence rate 
is $\exp(-br)$, which is even faster than the rate of $\exp(-br^{1/3})$ guaranteed by Theorem~\ref{th:Happrox}.

\subsection{3D-Example}

We consider the
crankshaft generated by NETGEN \cite{netgen} visualized in Figure~\ref{fig:crankshaft}.

\begin{figure}[h]
\begin{center}
\includegraphics[width=0.35\textwidth]{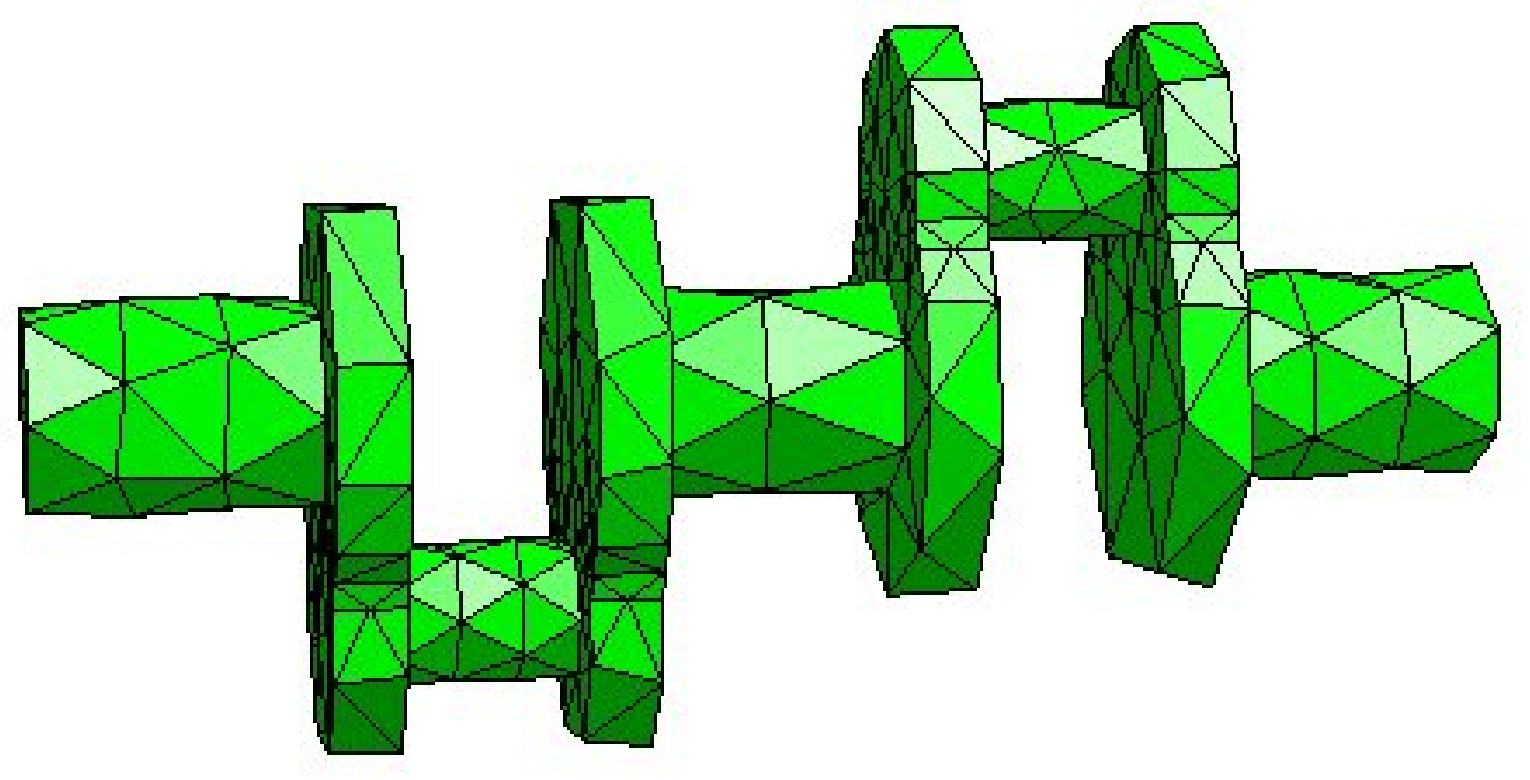}
\caption{\footnotesize Crankshaft domain}
\label{fig:crankshaft}
\end{center}
\end{figure}

In Figure \ref{fig:3DBEM}, we  compare the decrease of 
$\norm{\mathbf{I}-\mathbf{V}\mathbf{W}_{\mathcal{H}}}_2$ with the increase 
in the block-rank for a fixed number $N = 27,968$ of degrees of freedom, where 
the largest block of $\mathbf{W}_{\mathcal{H}}$ has a size of $3,496^2$. Table \ref{tab:3DBEM} 
shows the storage requirement for the matrix $\mathbf{W}_{\mathcal{H}}$ and the compression rates.
We mention that $\norm{\mathbf{W}_{\mathcal{H}}}_2 = 41.3$.

\begin{figure}[hbt]
\begin{minipage}{.50\linewidth}
\centering
\psfrag{Error}[c][c]{%
 \footnotesize  Error}
\psfrag{Block rank r }[c][c]{%
 \footnotesize  Block rank r}
\psfrag{asd}[l][l]{\footnotesize $\exp(-2.5\, r^{1/2})$}
\psfrag{jkl}[l][l]{\footnotesize $\norm{I-VW_{\mathcal{H}}}_2$}
\includegraphics[width=0.90\textwidth]{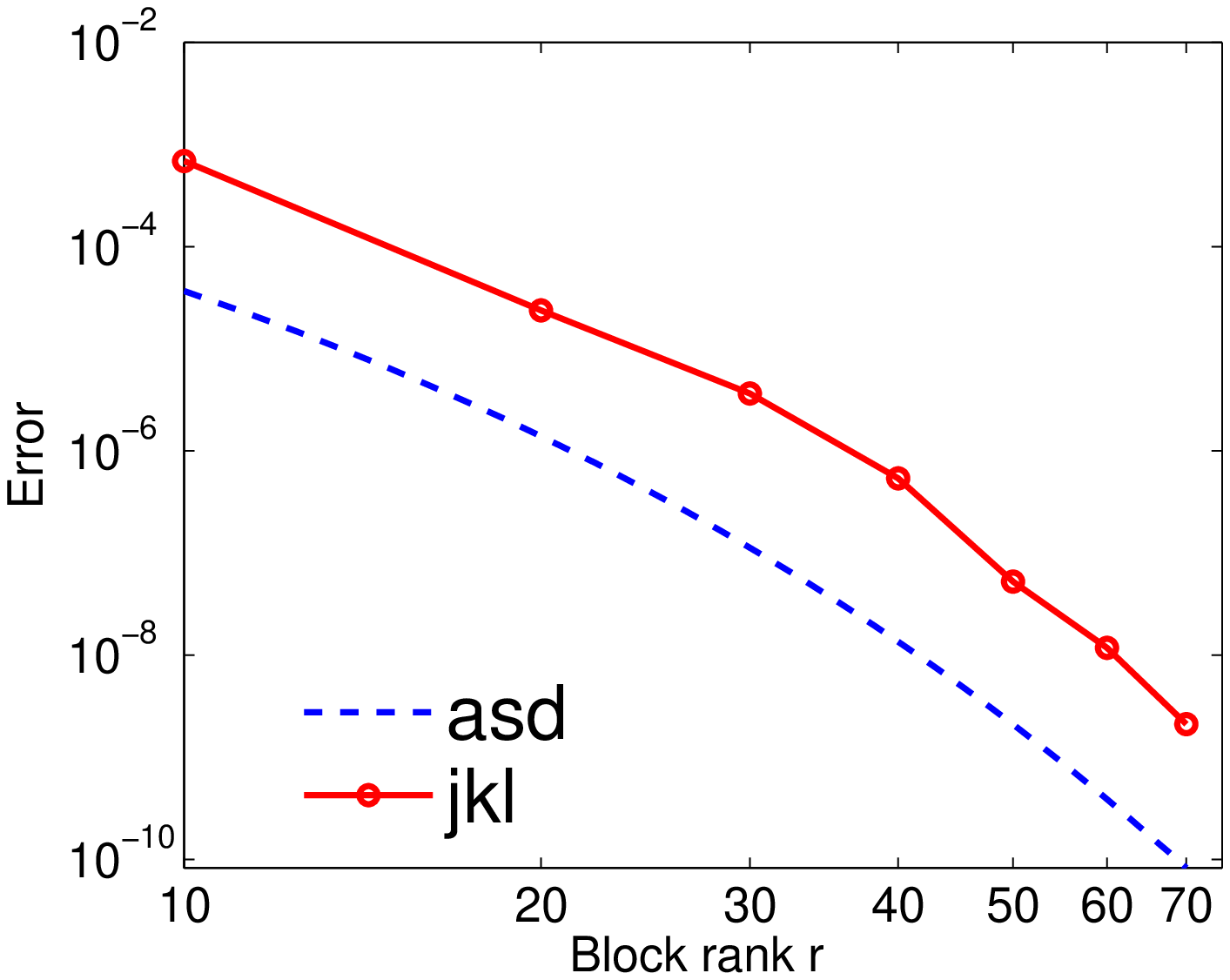}
\caption{\footnotesize Exponential convergence in block rank}
\label{fig:3DBEM}
\end{minipage}
\begin{minipage}{.49\linewidth}
\vspace{6mm}
\centering
\begin{tabular}{|c|c|c|}
\hline & Storage & Compressed \\
 $r$ & $\mathbf{W}_{\mathcal{H}}$ (MB) & to (\%)\\ 
\hline 10 & 1298 & 21.8 \\ 
\hline 20 & 1465 & 24.6 \\ 
\hline 30 & 1633 & 27.4 \\ 
\hline 40 & 1800 & 30.2 \\ 
\hline 50 & 1968 & 33.0 \\ 
\hline 60 & 2135 & 35.8 \\ 
\hline 70 & 2303 & 38.6 \\ 
\hline\end{tabular}
\vspace{2mm}
\captionof{table}{\footnotesize Storage (MB) for $\mathbf{W}_{\mathcal{H}}$
\label{tab:3DBEM} }
\end{minipage}
\end{figure}

Comparing the results with our theoretical bound from Theorem~\ref{th:Happrox},
 we empirically observe a rate of $\exp(-br^{1/2})$ instead
of $\exp(-br^{1/4})$.

\bibliography{bibliography_1}{}
\bibliographystyle{amsalpha}

\end{document}